\documentclass[11pt, a4paper]{amsart}

\usepackage{amsfonts,amsmath,amssymb, amscd,fullpage,mathtools}
\usepackage[all]{xy}
\usepackage{enumitem}
\usepackage{xcolor} 
\usepackage[utf8]{inputenc}
\usepackage[T1]{fontenc}

\newtheorem{theorem}{Theorem}[section]
\newtheorem{lemma}[theorem]{Lemma}
\newtheorem{definition}[theorem]{Definition}
\newtheorem{proposition}[theorem]{Proposition}
\newtheorem{corollary}[theorem]{Corollary}

\theoremstyle{definition}

\newtheorem{remark}[theorem]{Remark}
\newtheorem{example}[theorem]{Example}

\newtheorem*{notation}{Notation}

\newcommand\pf{\begin{proof}}
\newcommand\epf{\end{proof}}

\definecolor{myred}{rgb}{0.75,0,0}
\definecolor{mygreen}{rgb}{0,0.5,0}
\definecolor{myblue}{rgb}{0,0,0.65}

\newcommand\Oo{\mathcal{O}}
\newcommand\C{\mathcal{C}}
\newcommand\M{\mathcal{M}}
\newcommand\Z{\mathbb{Z}}

\newcommand\eps{\varepsilon}

\newcommand\co{\operatorname{co}}

\DeclareMathOperator{\Hom}{Hom}

\DeclareMathOperator{\id}{id}

\DeclareMathOperator{\ob}{ob}

\numberwithin{equation}{section}

\newcommand\op{\operatorname{op}}
\newcommand\N{\mathbb{N}}
\newcommand\ot{\otimes}

\DeclareMathOperator{\can}{can}

\def\lg{\langle}
\def\rg{\rangle}
\def\sl{\mathfrak{sl}}
\def\g{\mathfrak{g}}
\def\m{\mathfrak{m}}

\def\tr{\triangleright}
\def\tl{\triangleleft}
\def\tb{\bowtie}

\def\btr{\blacktriangleright}
\def\btl{\blacktriangleleft}
\newcommand{\btb}{\mathrel{\blacktriangleright\mspace{-3mu}\blacktriangleleft}}

\def\ttr{\overset{{\text{\tiny{t}}}}{\tr}}
\def\ttl{\overset{{\text{\tiny{t}}}}{\tl}}
\newcommand{\ttb}{\mathrel{\ttr\mspace{-2mu}\ttl}}

\def\hr{\rightharpoonup}
\def\hl{\leftharpoonup}

\allowdisplaybreaks

\newcommand{\GL}{\operatorname{GL}}

\title{Hopf-Galois objects over bicrossed product Hopf algebras and twisting maps}

\author{Julien Bichon}
\address{J. Bichon:	Universit\'e Clermont Auvergne, CNRS, LMBP, F-63000 CLERMONT-FERRAND, FRANCE}
\email{julien.bichon@uca.fr}

\author{Agust\'in Garc\'ia Iglesias}
\address{A. Garc\'ia Iglesias: Facultad de Matem\'atica, Astronom\'ia, F\'isica y Computaci\'on (FaMAF), 
	Medina Allende S/N, 
	Universidad Nacional de C\'ordoba,
	Ciudad Universitaria, C\'ordoba (X5000HUA),
	Argentina. }
\email{agustingarcia@unc.edu.ar}

\begin{document}

%\comj{IN RED:} Your comments (either written on the file or spoken at Marburg)

%\coma{IN BLUE:} The resolution of those comments, plus some editing. You can turn this back to black if you agree, and then delete the corresponding red part.

%\comy{IN GREEN:} Some guidelines for your reading, to be deleted.

\bigbreak

\begin{abstract}
	We describe Hopf-Galois objects over bicrossed product Hopf algebras. More precisely, we show that any right Hopf-Galois object over a bicrossed product of Hopf algebras is obtained from Hopf-Galois objects over the two factors and a certain twisting map, while the unique Hopf algebra making it into a bi-Galois object is again a bicrossed product.
%This note is a sequel to our previous paper \cite{biga}, where it was noticed that the first Weyl algebra has a natural Hopf-Galois structure, in the sense that it is a Hopf-Galois object over some Hopf algebra. We extend this remark to the algebra of differential operators on any affine group scheme. More generally, we discuss when the smash product of a Hopf-Galois object by a Hopf algebra still is a Hopf-Galois object.
\end{abstract}

\maketitle

\section{Introduction}\label{sec:intro}

%This note is a sequel to \cite{biga}, where it was noticed that the first Weyl algebra has a natural Hopf-Galois structure, in the sense that it is a Hopf-Galois object over some Hopf algebra. We extend this remark to the algebra of differential operators on any algebraic group.

Hopf-Galois extensions, introduced in \cite{kt}, are the natural analogues in noncommutative algebra of principal homogeneous spaces or torsors in algebraic geometry. The case with trivial base, called a Hopf-Galois object,  is already of great interest, because of its use in the study of tensor categories of comodules \cite{ulb,sc1}, classification questions for pointed Hopf algebras \cite{aagmv,ag}, homological algebra questions \cite{bic,yu,wyz,bic22}, or for example, with the recent connection found with quantum information theory \cite{bcehpsw}.

The question of the classification of Hopf-Galois objects over a given Hopf algebra is thus an important one and has been much studied in the past 30 years. While the case of group algebras or of enveloping algebras of Lie algebras is easily expressed in terms of second cohomology groups calculations, there is no general scheme, and even the case of coordinate algebras on affine algebraic groups is yet not fully understood, see \cite{gel} for several particular classes. 

The starting point  of this paper was the fact that the Weyl algebra $A_1(k)$ is a Hopf-Galois object  over the polynomial algebra $k[x,y]$, that we have noticed  in \cite{biga}, but  was certainly well-known before, and at least was known from \cite[Example 4.8]{gel}. Since $A_1(k)$ is, in characteristic zero, the algebra of differential operators on the additive algebraic group $k$, this leads to the question whether algebras of differential operators on affine algebraic groups always are Hopf-Galois objects. We observe (Proposition \ref{prop:diffG}) that this is indeed the case, combining a classical result of Heyneman-Sweedler \cite{hs} together with the generalized quantum double construction of Doi-Takeuchi \cite{dt94}. It results, for an affine group scheme $G$, that ${\rm Diff}(G)$, the algebra of differential operators on $G$,  is a right Hopf-Galois object over the tensor product Hopf algebra $\Oo(G)^{\rm cop}  \otimes \mathcal D(G)$, where  $\Oo(G)$ is the commutative Hopf algebra corresponding to $G$ and the algebra $\mathcal D(G)\subset \Oo(G)^*$ is the Hopf  algebra of distributions on $G$ (see Section \ref{sec:hhp} for more details).

%The above mentionned result of Heyneman-Sweedler says that for an affine group scheme $G$, then ${\rm Diff}(G)$, the algebra of diffrential operators on $G$ is isomorphic to a smash product algebra
%\[  {\rm Diff}(G)  \simeq \Oo(G) \mathcal \#\mathcal D(G) \]
%where $\Oo(G)$ is the commutative Hopf algebra corresponding to $G$, the algebra $\mathcal D(G)\subset \Oo(G)^*$ is the Hopf  algebra of distributions on $G$, and  the smash product is obtained via the he pairing coming from the inclusion $\mathcal D(G) \subset \Oo(G)^*$.

Hopf-Galois objects over tensor products of Hopf algebras, or more generally over generalized Drinfeld doubles, have been described by Schauenburg \cite{sctaft, sch99}. Therefore the natural next step seems to study the case of a general bicrossed product of Hopf algebras (the generalized Drinfeld double case being the one when the bicrossed product comes from a pairing), and this is precisely the contribution of this paper. We show that the Hopf-Galois objects over a Hopf algebra that is a bicrossed product of two Hopf algebras are described by Hopf-Galois objects over the two factors together with certain twisting maps.
 Our results and methods are illustrated by a complete study of a bicrossed product that is not produced by a pairing.

%To conclude this introduction, we would like to mention that the question whether an algebra has a Hopf-Galois structure can have some interest outside the world of Hopf algebraists. For example, recall that for the Hochschild cohomology of a Hopf algebra $H$ can be expressed simply using ${\rm Ext}$ on the category  of left (or right) $H$-modules, and that this simplifies several constructions and computations in the homological study of the Hopf algebra $H$. For a general Hopf-Galois object $A$ over $H$, there is no such result, but the Hochschild cohomology of $A$ can be expressed in terms of that of $H$, and hence in terms of ${\rm Ext}$ on the category of $H$-modules (this follows from  Stefan's spectral sequence \cite{ste}, but can be proved quite directly as well \cite{bic}). As a very useful  illustration, we wish to mention Yu's result \cite{yu}, that states that if $H$ is a twisted Calabi-Yau algebra of dimension $d$, then so is $A$.

The paper is organized as follows. We start in Section \ref{sec:hhp} by recalling some basic notions about Hopf Galois objects and pairings, from which we deduce, via Proposition \ref{prop:smashgaloishopfhopf}, that for an affine group scheme $G$, the algebra ${\rm Diff}(G)$  is a Hopf-Galois object. In Section \ref{sec:twistings}  we recall the notions of twisting map for a pair of algebras and of bicrossed product of Hopf algebras. We then discuss the relevant notion of (skew) pairing in the bicrossed product context and generalize Proposition \ref{prop:smashgaloishopfhopf} to this context (Proposition \ref{prop:skewbicrossed}),  obtaining in this way an interesting source of  Hopf-Galois objects for Hopf bicrossed producs.
	In Section \ref{sec:hopf-galois-bicrossed}, we study the relation  between the bicrossed product of two Hopf algebras and the appropriate twisted tensor products of a pair of Hopf-Galois objects for them. This leads to a description of the right  Hopf-Galois objects (Theorem \ref{thm:bicrossedtwist1}) and of the corresponding Hopf algebras over which these are left Hopf-Galois (Theorem \ref{thm:bicrossedtwist2}), which are shown to be bicrossed products as well.
	Section \ref{sec:exbicrossed} provides a detailed analysis of an example of bicrossed product Hopf algebra not arising from a pairing.
	In the final Section \ref{sec:semi} we restrict our  previous results to the case of a the semi-direct product of Hopf algebra. This allows us to provide a more precise presentation of the Hopf-Galois objects in this case in Theorem \ref{thm:galoissemidirect}. We discuss in particular the case of unrolled Hopf algebras. %In each section we include examples to illustrate these constructions.

\smallskip

We work over a fixed base field $k$, over which all the tensor products are taken. We write $k^\times$ for the units in $k$.

\section{Hopf-Galois objects and pairings}\label{sec:hhp}

\subsection{Hopf-Galois objects}\label{sub:hopfgalois}
Recall that a (right) \textsl{Hopf-Galois object} $A$ over a Hopf algebra $H$ \cite{kt} is a (right) comodule algebra $A$ over $H$ with trivial coinvariants $A^{\co H}=k$ and such that the 
{\it Galois map} 
\[
\can\colon A\ot A\to A\ot H, \qquad \can(a\ot b)=ab_{(0)}\ot b_{(1)}, \ 	a,b\in A
\]
is bijective. Along the text, we will consider the map 
\begin{align}\label{eqn:kappa}
%\begin{split}
\kappa : H  & \longrightarrow A \otimes A, &
h &\longmapsto h^{(1)} \otimes h^{(2)}\coloneqq {\rm can}^{-1}(1_A\otimes h);
%\end{split}
\end{align}
hence we have that:
\begin{align}\label{eqn:kappa-on-elements}
h^{(1)}(h^{(2)})_{(0)} \otimes (h^{(2)})_{(1)}=1\ot h, \qquad h\in H.
\end{align}
As well, we have, see
\cite[Remark 3.4 2\,(b)]{schn}:
\begin{align}\label{eqn:canonical-schneider}
a_{(0)}(a_{(1)})^{(1)} \otimes (a_{(1)})^{(2)}=1\ot a, \qquad a\in A.
\end{align}
Similarly, there is a notion of left Hopf-Galois object, and a notion of bi-Galois object \cite{sc1}. Schauenburg \cite[Theorem 3.5]{sc1} has shown that, given a right Hopf-Galois object $A$ over $H$, there exists a unique, up to isomorphism, Hopf algebra $L=L(A,H)$ such that $A$ is an $L$-$H$-bi-Galois object. The Hopf algebra $L(A,H)$ is defined by $L(A,H)= (A\otimes A^{\op})^{\co H}$, and the left coaction $\lambda_A : A \to L \otimes A$ is given by $\lambda_A(a) = a_{(0)}\otimes \kappa(a_{(1)})$.

\subsection{Hopf $2$-cocycles} \label{subsec:cocycle} Let $H$ be a Hopf algebra. Recall that  a (left) \textsl{$2$-cocycle} on $H$ is a convolution invertible linear map $\sigma  : H \otimes H \to k$ such that
 $\sigma(x,1)=\varepsilon(x)=\sigma(1,x)$ for any $x\in H$, and we have, for any $x,y,z\in H$,
\[ \sigma(x_{(1)}, y_{(1)})\sigma(x_{(2)}y_{(2)},z)= \sigma(y_{(1)},z_{(1)})\sigma(x,y_{(2)}z_{(2)}).\]
There are several algebras naturally associated to such a $2$-cocycle.

$\bullet$ The algebra $_{\sigma}\!H$ has $H$ as underlying vector space, and product defined by
\begin{align}\label{eqn:cocycle-def-galois}
 x . y = \sigma(x_{(1)},y_{(1)}) x_{(2)}y_{(2)}.
\end{align} 
The comultiplication of $H$ then endows  $_{\sigma}\!H$ with the structure of a right $H$-comodule algebra, making it into a right $H$-Galois object \cite{bm,dt86}. Notice that the $2$-cocycle condition is equivalent to the fact that the above product is associative and has $1$ as unit.

$\bullet$ The Hopf algebra $H^\sigma$ \cite{doi} has $H$ as underlying coalgebra, and product defined by
\begin{align}\label{eqn:cocycle-def}
x . y = \sigma(x_{(1)},y_{(1)}) \sigma^{-1}(x_{(3)},y_{(3)})x_{(2)}y_{(2)}.
\end{align}
The comultiplication of $H$ then endows  $_{\sigma}\!H$ with the structure of a left $H^\sigma$-comodule algebra, making it into a left $H^\sigma$-Galois object, and hence an $H^\sigma$-$H$-bi-Galois object \cite{sc1}.

\subsection{Cogroupoids}\label{sec:cogroupoids}

We briefly recall here part of the terminology and results from \cite{bic} on cogroupoids. We refer the reader to loc.~cit.~for details.

We start by recalling that a cocategory consists of a set of objects $\ob\C$, a $k$-algebra $\C(X,Y)$ for each $X,Y\in\ob\C$ and algebra morphisms, for every $X,Y,Z\in\ob\C$
\begin{align*}
\Delta_{X,Y}^Z\colon & \C(X,Y)\to \C(X,Z)\ot \C(Z,Y), & \eps_X\colon & \C(X,X)\to k
\end{align*}  
satisfying compatibility axioms that mimic those of a coassociative coalgebra, namely 
\begin{align*}
(\Delta_{X,Z}^T\ot \id_{\C(Z,Y)})\circ \Delta_{X,Y}^Z&=(\id_{\C(X,T)}\ot \Delta_{T,Y}^Z)\circ \Delta_{X,Y}^T, \\
(\id_{\C(X,Y)}\ot\eps_Y)\circ \Delta_{X,Y}^Y&=\id_{\C(X,Y)}=(\eps_X\ot \id_{\C(X,Y)}\ot\eps_Y)\circ \Delta_{X,Y}^X.
\end{align*}
for every $X,Y,Z,T\in \ob \C$. The cocategory $\C$ is called connected if $\C(X,Y)\neq 0$ for all $X,Y\in\ob \C$.
In turn, a cogroupoid is a cocategory $\C$ together with linear maps
\[
S_{X,Y}\colon \C(X,Y)\to \C(Y,X), \qquad \forall\, X,Y\in \ob\C
\]
subject to the natural axioms of the antipode on a Hopf algebra. 
It follows from the definitions that a bialgebra is a cocategory with a single object; and $\C(X,X)$ is a Hopf algebra for any object $X$ in a cogroupoid $\C$. 
We use Sweedler's notation $\Delta_{X,Y}^Z(a^{X,Y})=a^{X,Z}_{(1)}\ot a^{Z,X}_{(2)}$ for an element $a^{X,Y}\in\C(X,Y)$.

If $\C$ is a connected cogroupoid, then $\C(X,Y)$ is a $\C(X,X)-\C(Y,Y)$ bi-Galois object for all $X,Y\in\ob\C$, see \cite[Proposition 2.8]{bic}. Conversely, for any Hopf algebra $H$ and a left $H$-Hopf-Galois object $A$, there exists a cogroupoid $\C$ and $X,Y\in\ob\C$ so that $H=\C(X,X)$ and $A=\C(X,Y)$, \cite[Theorem 3.11]{bic}.
We shall also use \cite[Lemma 2.14]{bic}, which shows that $\Delta_{X,Y}^Z$ induces a $\C(X,X)-\C(Y,Y)$-bicomodule algebra isomorphism
\begin{align}\label{eqn:bicom-iso}
	\C(X,Y)\simeq \C(X,Z)\square_{\C(Z,Z)}\C(Z,Y).
\end{align}

If $\C$ is a connected cogroupoid, then \cite[Theorem 2.12]{bic} presents a $k$-linear equivalence of monoidal categories
\[
\M^{\C(X,X)}\simeq^{\ot} \M^{\C(Y,Y)}, \qquad V\mapsto V\square_{\C(X,X)}\C(X,Y)
\]
for every pair $(X,Y)$ of objects in $\C$. The monoidal constraint is given by the isomorphism
\begin{align}\label{eqn:monoidal}
\begin{split}
(V\square_{\C(X,X)}\C(X,Y))\ot (W\square_{\C(X,X)}\C(X,Y))&\to (V\ot W)\square_{\C(X,X)}\C(X,Y)\\
(\sum_i v_i\ot a_i^{X,Y})\ot (\sum_j w_j\ot b_j^{X,Y})&\mapsto \sum_{i,j} v_i\ot  w_j\ot a_i^{X,Y}b_j^{X,Y}
\end{split}
\end{align}
for all $V,W\in\M^{\C(X,X)}$, see \cite[Remark 2.16]{bic}.

Subcocategories, respectively subcogroupoids, $\C'\subseteq \C$ are defined as expected. 
In particular, we make the following remark, which is clear from the definitions and \cite[Proposition 2.13]{bic}.

\begin{lemma}\label{lem:subco}
	Let $\C$ be a connected cogroupoid, $X\in\ob \mathcal C$ and let $H\subseteq \C(X,X)$ be a Hopf subalgebra. Consider, for $Y,Z \in \ob \mathcal C$, the algebra
\[\C_H(Y,Z)=\{a^{Y,Z}\in\C(Y,Z) : a^{Y,X}_{(1)}\ot a^{X,X}_{(2)}\ot a^{X,Z}_{(3)}  \in \C(Y,X)\ot H\ot \C(X,Z)\}
\]
	This defines a connected subcogroupoid  $\C_H$  of $\C$ (with $\ob\C_H=\ob\C$ and the natural restrictions of the maps $\Delta, \eps$ and $S$).
\end{lemma}

\begin{remark}
	Notice that 
	\[\C_H(Y,Z) \simeq \C(Y,X)\square_{\C(X,X)} H\square_{\C(X,X)}\C(X,Z)
	\] and in particular
	\begin{align}\begin{split}\label{eqn:cotensor}
			\C_H(X,Y)&\simeq H\square_{\C(X,X)}\C(X,Y), \qquad\qquad  \C_H(Y,X)\simeq \C(Y,X)\square_{\C(X,X)} H.
	\end{split}
	\end{align}
	As well, we have  $\C_H(X,X)=H$ and
	\begin{align*} \label{eqn:cotensor}
		\C_H(X,Y)&=\{a^{X,Y}\in\C(X,Y) : a^{X,X}_{(1)}\ot a^{X,Y}_{(2)} \in  H\ot \C(X,Y)\}, \\
	\C_H(Y,X)&=\{a^{Y,X}\in\C(Y,X) : a^{Y,X}_{(1)}\ot a^{X,X}_{(2)} \in  \C(Y,X)\ot H\}.
	\end{align*}
\end{remark}

%\begin{lemma}\label{lem:subco}
%	Let $\C$ be a connected cogroupoid, $X,Y\in\ob \mathcal C$ and let $H\subseteq \C(X,X)$ be a Hopf %subalgebra. Consider the algebras
%	$\C_H(X,X)=H$, 
%	\begin{align*}
	%	\C_H(X,Y)&=\{a^{X,Y}\in\C(X,Y) : a^{X,X}_{(1)}\ot a^{X,Y}_{(2)} \in  H\ot \C(X,Y)\}, \\
%		\C_H(Y,X)&=\{a^{Y,X}\in\C(Y,X) : a^{Y,X}_{(1)}\ot a^{X,X}_{(2)} \in  \C(Y,X)\ot H\}, \\
	%	\C_H(Y,Y)&=\{a^{Y,Y}\in\C(Y,Y) : a^{Y,X}_{(1)}\ot a^{X,X}_{(2)}\ot a^{X,Y}_{(3)}  \in \C(Y,X)\ot H\ot \C(X,Y)\}.
	%\end{align*}
%	Then $\C_H$ is a connected subcogroupoid of $\C$, with $\ob\C_H=\{X,Y\}$ and the natural restrictions of the maps $\Delta, \eps$ and $S$.\hfill\qed
%\end{lemma}
%\begin{remark}
%	Notice that  
%	\begin{align}\begin{split}\label{eqn:cotensor}
	%		\C_H(X,Y)&\simeq H\square_{\C(X,X)}\C(X,Y), \qquad\qquad  \C_H(Y,X)\simeq \C(Y,X)\square_{\C(X,X)} H,\\
%			\C_H(Y,Y)&\simeq \C(Y,X)\square_{\C(X,X)} H\square_{\C(X,X)}\C(X,Y).
	%	\end{split}
%	\end{align}
%\end{remark}

\subsection{Pairings}\label{subsec:pairing}
Let $A$ and $U$ be Hopf algebras. Recall that a \textsl{pairing} between $A$ and $U$ consists of a linear map
$ \tau : A \otimes U \longrightarrow k$
such that for any $a,b \in A$, $x,y\in U$, we have:
\begin{align*}
	\tau(a,1)&=\varepsilon(a), & \tau(1,x)&=\varepsilon(x),\\
 \tau(ab,x) &=  \tau(a,x_{(1)})\tau(b,x_{(2)}) & 
\tau(a,xy) &= \tau(a_{(1)},x)\tau(a_{(2)},y) 
\end{align*}
A \textsl{skew-pairing} between $A$ and $U$ is a pairing between $A^{\rm cop}$ and $U$. 

Doi-Takeuchi have shown \cite[Proposition 1.5]{dt94} that if  $\tau : A \otimes U \to k$ is a  skew-pairing, then
\begin{align*}
 \hat{\tau} : A^{} \otimes U  \otimes A^{} \otimes U &\longrightarrow k, & 
a\otimes x \otimes b \otimes y & \longmapsto \tau(b,x) \varepsilon(a)\varepsilon(y)
\end{align*}
is a $2$-cocycle on $A\otimes U$.
We thus obtain, via the constructions of Subsection \ref{subsec:cocycle}, the algebra  $_{\hat{\tau}}\!(A^{}\otimes U)$  and the Hopf algebra $(A^{}\otimes U)^{\hat{\tau}}$. The Hopf algebra $(A\otimes U)^{\hat{\tau}}$ is denoted $D_\tau(A^{},U)$ and is called a \textsl{generalized Drinfeld double}. %This is consistent with the usual terminology, because $\tau$ is a skew-pairing between $A^{\rm cop}$ and $U$.

\subsection{Smash products} \label{sub:smash} Let $U$ be a Hopf algebra and let $A$ be an algebra. A left $U$-module algebra structure on $A$ consists of a linear map
\begin{align*}
 U \otimes A &\longrightarrow A &
x\otimes a &\longmapsto x.a
\end{align*}
making $A$ into a left $U$-module, and such that, for any $x\in U$ and $a,b \in A$, we have 
\[x.(ab)=(x_{(1)}.a) (x_{(2)}.b), \qquad x.1 =\varepsilon(x)1.\]  
If $A$ is left $U$-module algebra, the associated \textsl{smash product} algebra $A\#U$ is the algebra having $A\otimes U$ as underlying vector space, and whose product is defined by 
\[a\#x \cdot b\#y= a(x_{(1)}.b)\# x_{(2)}y.\]

\subsection{Hopf-Galois structure on smash products arising from pairings} \label{subsec:hopfhopf}
Let $A$ and $U$ be Hopf algebras and let
 $\tau : A \otimes U \longrightarrow k$ be a pairing. The pairing then provides a map
\begin{align*}
 U \otimes A & \longrightarrow A &
x\otimes a &\longmapsto x.a= \tau(a_{(2)},x)a_{(1)}
\end{align*}
that endows $A$ with the structure of a left $U$-module algebra structure. We thus can form the smash product algebra $A\#U$, that we denote $A\#_\tau U$ to keep track of $\tau$, whose product is:
\[a\#x \cdot b\#y= a(x_{(1)}.b)\# x_{(2)}y= \tau(b_{(2)},x_{(1)}) ab_{(1)}\# x_{(2)}y.\] 
 It is a simple verification that the above product coincides with the one on  $_{\hat{\tau}}\!(A^{\rm cop}\otimes U)$ as in Subsection \ref{subsec:cocycle}, and thus the considerations in subsections \ref{subsec:cocycle} and  \ref{subsec:pairing} yield the following result.

\begin{proposition}\label{prop:smashgaloishopfhopf}
 Let $A$ and $U$ be Hopf algebras and let $\tau : A \otimes U \longrightarrow k$ be a pairing. Then the smash product algebra $A\#_\tau U$ is a $D_\tau(A^{\rm cop},U)-(A^{\rm cop}\otimes U)$-bi-Galois object, whose left and right coactions are induced by the comultiplication of the tensor product coalgebra $A^{\rm cop}\otimes U$.
\end{proposition}

For the purpose of future generalizations, it is worth to record the equivalent skew-pairing version of the previous proposition. So again let $A$ and $U$ be Hopf algebras and let
$\tau : A \otimes U \longrightarrow k$ be a skew-pairing. The pairing then provides a map
\begin{align*}
U \otimes A & \longrightarrow A &
x\otimes a &\longmapsto x.a= \tau(a_{(1)},x)a_{(2)}
\end{align*}
that endows $A$ with the structure of a left $U$-module algebra structure. We thus form the smash product algebra $A\#U$, that we denote $A\#_\tau U$ to keep track of $\tau$ (viewing our skew-pairing as a pairing between $A^{\rm cop}$ and $U$, this is the previous $A^{\rm cop}\#_\tau U$), whose product is defined by
\[a\#x \cdot b\#y= a(x_{(1)}.b)\# x_{(2)}y= \tau(b_{(1)},x_{(1)}) ab_{(2)}\# x_{(2)}y.\] 
Again, the above product coincides with the one on  $_{\hat{\tau}}\!(A\otimes U)$ as in Subsection \ref{subsec:cocycle}, yielding, as before, the following result.

\begin{proposition}\label{prop:smashgaloishopfhopfskew}
	Let $A$ and $U$ be Hopf algebras and let $\tau : A \otimes U \longrightarrow k$ be a skew-pairing. Then the smash product algebra $A\#_\tau U$ is a $D_\tau(A,U)-(A\otimes U)$-bi-Galois object, whose left and right coactions are induced by the comultiplication of the tensor product coalgebra $A\otimes U$.
\end{proposition}

\subsection{Application to differential operators on affine group schemes}\label{diff}
The first Weyl algebra $A_1(k)= k\langle x,y \ | \ xy-yx=1\rangle $ can be described as the smash product $k[y]\#k[x]$, where $k[x]$ and $k[y]$ have the Hopf algebra structure making $x$ and $y$ primitive, and the smash product structure is obtained as in Subsection \ref{subsec:hopfhopf} via the pairing $\tau : k[x] \otimes k[y] \to k$ defined by $\tau(x,y)=1$.

The structure of Hopf-Galois object over $k[x,y]$ of $A_1(k)$  is a particular instance of Proposition \ref{prop:smashgaloishopfhopf}. The Weyl algebra is, in characteristic zero, the algebra of differential operators on the additive algebraic group $k$. This leads to the question whether algebras of differential operators on algebraic groups always are Hopf-Galois objects. As we shall see soon, a classical result of Heyneman-Sweedler \cite{hs} together with Proposition \ref{prop:smashgaloishopfhopf} provide a positive answer.

Let $G$ be an affine group scheme, with corresponding commutative  and finitely generated coordinate Hopf algebra $\Oo(G)$. The algebra of differential operators on $G$, denoted ${\rm Diff}(G)$, is defined as the algebra of differential operators on the commutative algebra $\Oo(G)$, in the usual manner, see e.g. \cite{McR}. Denote by $\mathcal D(G)$ the algebra of distributions on $G$ (this is $\mathcal D(\Oo(G))$ in \cite{hs}, see \cite{dg,jan} as well). This is a Hopf subalgebra of the Hopf dual $\Oo(G)^{\circ}$, isomorphic, in characteristic zero, with $U(\mathfrak g)$, the enveloping algebra of $\mathfrak g$, the Lie algebra of $G$. Now \cite[Theorem 2.4.5]{hs} establishes an algebra isomorphism 
\[ \Oo(G) \mathcal \#\mathcal D(G) \simeq {\rm Diff}(G) \]
where the smash product is associated to the pairing coming from the inclusion $\mathcal D(G) \subset \Oo(G)^\circ$. Combining Proposition \ref{prop:smashgaloishopfhopf} with the above isomorphism, we get:

\begin{proposition}\label{prop:diffG}
	Let $G$ be an affine group scheme. Then ${\rm Diff}(G)$ is a right Hopf-Galois object over the Hopf algebra $\Oo(G)^{\rm cop}  \otimes \mathcal D(G)$.
\end{proposition}

\section{Twisting maps and bicrossed products of Hopf algebras}\label{sec:twistings}

In this section we review the notions of a twisted tensor product of algebras and of a bicrossed product of Hopf algebras, and we provide the first generalization of Proposition \ref{prop:smashgaloishopfhopfskew}.

\subsection{Twisting maps} \label{sub:twist} Let $A$, $R$ be algebras. Recall that a \textsl{twisting map} for $A$, $R$ consists of a linear map
$$\theta : R\otimes A \to A \otimes R$$
such that the map
$$A\otimes R \otimes A \otimes R \xrightarrow{\ {\rm id}_A\otimes \theta \otimes {\rm id}_R \ } A\otimes A \otimes R \otimes R 
\xrightarrow{\ m_A\otimes m_R\ } A \otimes R$$
makes $A\otimes R$ into an associative algebra, with $1\otimes 1$ as a unit. The resulting algebra can be denoted $A\otimes_\theta R$ or  $A\#_\theta R$, and is called a \textsl{twisted tensor product of $A$ and $R$}. This construction has been studied in many papers, see \cite{vdvk,caen} for example, where the following equivalent conditions are established, for $a,a'\in A$, $r,r'\in R$:
\begin{align}\label{eqn:twistconditions}
\begin{split}
\theta\circ(m_R\ot\id)(r\ot r'\ot a)&=(\id\ot \,m_R)\circ(\theta\ot\id)\circ(\id\ot\, \theta)(r\ot r'\ot a),\\
\theta\circ(\id\ot\, m_A)(r\ot a\ot a')&=(m_A\ot \id)\circ(\id\ot\,\theta)\circ (\theta\ot \id)(r\ot a'\ot a),\\
\theta(1\otimes a)&=a\ot 1, \qquad \theta(r\otimes 1)=1\ot r.
\end{split}
\end{align}
\begin{example}
	The very first non-trivial example of a twisted tensor product comes from smash products: if $U$ is a Hopf algebra and $A$ is a $U$-module algebra, then the map
	\begin{align*}
	U \otimes A &\longrightarrow A \otimes U &
	x \otimes a & \longmapsto x_{(1)}.a \otimes x_{(2)}
	\end{align*}
	is a twisting map, and the resulting twisted tensor product is the smash product algebra of Subsection \ref{sub:smash}.
\end{example}

The twisted tensor product generalizes the braided tensor product of algebras (typically the braided tensor product of Yetter-Drinfeld algebras) as well. If we write
\[
\theta(r\ot a)=r\hr a\ot r\hl a
\]
then the above conditions become $1\hr a\ot 1\hl a=a\ot 1, r\hr 1\ot r\hl 1=1\ot r$ and
\begin{align}\label{eqn:theta-explicit}
\begin{split}
(rr'\hr a)\ot (rr'\hl a)&=  r\hr (r'\hr a)\ot (r\hl (r'\hr a))(r'\hl a),\\
(r\hr aa')\ot (r\hl aa')&=( r\hr a)( (r\hl a)\hr a')\ot (r\hl a)\hl a',
\end{split}
\end{align}
for any $r,r'\in R$, $a,a'\in A$. These equations will become handy further on.

Conversely, starting with an algebra $E$ having two subalgebras such that the restricted multiplication $A \otimes R \to E$ is bijective, then there exits a twisting map $\theta : R\otimes A \to A \otimes R$ such that $E\simeq A\#_\theta R$, where the twisting map is obtained as the composition of the restricted multiplication $R \otimes A \to J$ with the flip map. See \cite[Theorem 2.10]{caen} for example.

It is in general a difficult task to classify the twisted tensor product of two algebras, see e.g. \cite{aggv} for a recent paper on this question.

\subsection{Bicrossed products of Hopf algebras}\label{subsec:bicrossed} 
 Recall that a Hopf algebra $E$ factors through two Hopf subalgebras $\iota\colon H\hookrightarrow E$ and $j:U\hookrightarrow E$ if the multiplication $m\circ (\iota\ot j)\colon H\ot U\stackrel{\simeq }{\longrightarrow} E$ is bijective. This is equivalent, see \cite[Theorem 7.2.3]{maj}, to have a pair $(H,U)$ of Hopf algebras together with coalgebra maps 
\[
\tl\colon U\ot H\to U, \qquad \tr\colon U\ot H\to H,
\]
so that $(U,\triangleleft)$ is a left $H$-module coalgebra, $(H,\triangleright)$ is a right $U$-module coalgebra and the following compatibilities hold, for every $x,y\in U$, $a,b\in H$:
\begin{align}
	\begin{split}\label{eqn:compatibilities-matchedpair}
x\tr 1_H&=\eps_U(x)1_H, \qquad 1_U\tl a=\eps_H(a)1_U,\\
x\tr(ab)&=(x_{(1)}\tr a_{(1)})((x_{(2)}\tl a_{(2)})\tr b),\\
(yx)\tl a&=(y\tl (x_{(1)}\tr a_{(1)}))(x_{(2)}\tl a_{(2)}),\\
y_{(1)}\tl a_{(1)}\ot y_{(2)}\tr a_{(2)}&= y_{(2)}\tl a_{(2)}\ot y_{(1)}\tr a_{(1)}.
	\end{split}
\end{align}

A quadruple $(H,U,\triangleright, \triangleleft)$ as above is called a \textsl{matched pair of Hopf algebras} and  the \textsl{bicrossed product Hopf algebra} $H\tb U\coloneqq E$, often called a double crossed (as pointed out in \cite{ago1}, the name bicrossed product fits better with the bicrossed product of groups), is the coalgebra $H\ot U$ with multiplication
\[
(a\tb x)(b\tb y)=a(x_{(1)}\tr b_{(1)})\tb (x_{(2)}\tl b_{(2)})y,  \qquad a,b\in H, x,y\in U.
\]
In particular, this defines a twisting map and a coalgebra morphism:
\begin{align}\label{eqn:twisting-bicrossed}
	\omega_{\tb}\colon U\ot H\to H\ot U, \qquad x\ot h\mapsto x_{(1)}\tr h_{(1)}\ot x_{(2)}\tl h_{(2)}.
\end{align}

We say that one of the actions, say $\tl$, is trivial, if $x\tl a=\eps(a)x$ for all $x\in U,a\in H$. 

\medbreak

\begin{example}\label{ex:basicbicrossed}
Examples of this construction are:
\begin{enumerate}
	\item The tensor product Hopf algebra $H\ot U$, where both actions $\tl$ and $\tr$ are trivial.
	\item Smash products of Hopf algebras $H\# U$, where one of the actions, namely $\tl\colon U\ot H\to U$ is trivial. We will come back to this situation in more detail in Section \ref{sec:semi}.
	\item Group algebras $k\Gamma$, arising from a matched pairs of groups $(F,G,\triangleright,\triangleleft)$, so that $G,F\leq \Gamma$, $G\cap F=\{e\}$ and $\Gamma=GF$.
	\item Bicrossed products $U(\Xi)\simeq U(\g)\tb U(\m)$ associated to a matched pair of Lie algebras $(\g,\m)$ as in \cite[Definition 8.3.1]{maj}. Here $\Xi=\g\oplus \m$ as vector spaces and the bracket $[\,,]_\Xi$ is written in terms of $[\,,]_\g, [\,,]_\m$ and the actions defining the matched pair.
	\item The generalized Drinfeld doubles $D_\tau(A,U)=(A\otimes U)^{\hat{\tau}}$  associated to skew pairings $\tau\colon A\ot U\to k$ as in \S \ref{subsec:pairing}.
	Here, $D_\tau(A,U)\simeq A\tb U$ with 
	\begin{align*}
		(1\tb x)(a\tb 1)=\tau(a_{(1)},x_{(1)})a_{(2)}x_{(2)}\tau^{-1}(a_{3},x_{(3)}), \ x\in U, a\in A.
	\end{align*}
	In this setting we write $A\tb_\tau U$ to remark the existence of this pairing.
\end{enumerate} 
\end{example}

We have not been able to found in the literature an example not belonging to any of these groups. In particular, we observe that in the ongoing program of classifying bicrossed products of Hopf algebras by Agore et.~al., see \cite{ago1,ago2,ago3}, all the examples fall into one of the categories above (most notably into the fourth). The same holds for \cite{bont} and \cite{lu}. The Galois objects in this setting have been described in \cite{sctaft}. We present in Section \ref{sec:exbicrossed} below a case essentially different to those above, and that corresponds to the scope of the present article.

First, we remark that when the antipodes $S_H$ and $S_U$ of $H$ and $U$ are invertible (or equivalently when $S_E$ is), then this construction can be flipped over, in the sense that the roles of $H$ and $U$ can be interchanged. This is the content of the next result.

%\comy{New lemma, that helps with the computation of $\theta_\psi^{-1}$ in \eqref{eqn:theta-invertible}:}

\begin{lemma}\label{lem:reverse}
	Let $(H,U,\triangleright, \triangleleft)$ be a matched pair of Hopf algebras. Assume that the antipodes $S_H$ and $S_U$ are bijective. Then there are  coalgebra maps 
	\[
	\ttl\colon H\ot U\to H, \qquad \ttr\colon H\ot U\to U
	\]
	so that $(H,\ttl)$ is a right $U$-module coalgebra, $(U,\ttr)$ is a left $H$-module coalgebra and the  compatibilities \eqref{eqn:compatibilities-matchedpair} hold for $\ttl$ and $\ttr$. More explicitly,
	\begin{align*}
		a\ttr x&=S_U(S_U^{-1}(x)\tl S_H^{-1}(a)), & a\ttl x&=S_H(S_U^{-1}(x)\tr S_H^{-1}(a)).
	\end{align*}
The bicrossed product Hopf algebras coincide, as $H\tb U\simeq U\ttb H$ and the following identities hold:
	\begin{align}
		\begin{split}\label{eqn:reverse-match}
		ax&=[(a_{(1)}\ttr x_{(1)})\tr (a_{(3)}\ttl x_{(3)})][(a_{(2)}\ttr x_{(2)})\tr (a_{(4)}\ttl x_{(4)})],\\
		xa&=[(x_{(1)}\tr a_{(1)})\ttr (x_{(3)}\tl a_{(3)})][(x_{(2)}\tr a_{(2)})\ttl (x_{(4)}\tl a_{(4)})].
	\end{split}
	\end{align}
\end{lemma}
\pf
It is clear that $(H,\ttl)$ is a right $U$-module coalgebra, $(U,\ttr)$ is a left $H$-module coalgebra and the compatibilities \eqref{eqn:compatibilities-matchedpair} hold, since these facts follow from the corresponding properties of $\tl$ and $\tr$. As well, observe that, for $b=S_H^{-1}(a)$ and $y=S_U^{-1}(x)$, we have:
\begin{align*}
ax&=S_H(b)S_U(y)=S_E(yb)=S_E((y_1\tr b_1)(y_2\tl b_2))=S_U(y_2\tl b_2)S_H((y_1\tr b_1))\\
&=S_U(S_U^{-1}(x_{(1)})\tl S_H^{-1}(a_{(1)}))S_H(S_U^{-1}(x_{(2)})\tr S_H^{-1}(a_{(2)}))=(a_{(1)}\ttr x_{(1)})(a_{(2)}\ttl x_{(2)}),
\end{align*}
which shows $H\tb U \simeq  U\ttb H$. Finally, to check \eqref{eqn:reverse-match}, we use the definition of $H\tb U$ in:
\begin{align*}
ax&=(a_{(1)}\ttr x_{(1)})(a_{(2)}\ttl x_{(2)})=[(a_{(1)}\ttr x_{(1)})\tr (a_{(3)}\ttl x_{(3)})][(a_{(2)}\ttr x_{(2)})\tl (a_{(4)}\ttl x_{(4)})].
\end{align*}
The other identity follows from the definition of $H\ttb U$:
\begin{align*}
xa&=(x_{(1)}\tr a_{(1)})(x_{(2)}\tl a_{(2)})=[(x_{(1)}\tr a_{(1)})\ttr (x_{(3)}\tl a_{(3)})][(x_{(2)}\tr a_{(2)})\ttl (x_{(4)}\tl a_{(4)})].
\end{align*}
The lemma follows.
\epf

\subsection{Skew-pairings over bicrossed products} We now provide a generalization of skew pairing to the setting of bicrossed products. 

\begin{definition}\label{def:skew-pairing}
Let $(H,U,\tr,\tl)$ be a matched pair of Hopf algebras and let $E\coloneqq H\tb U$ be the corresponding bicrossed product. A skew $H\tb U$-pairing is a convolution invertible linear map $\tau : H \otimes U \to k$ 
such that,  for all $g,h\in H$, $x,y\in U$:
\begin{align}
\begin{split}\label{eqn:tauskew}
\tau(h,1)=\varepsilon(h),  \ & \tau(1,x)= \varepsilon(x),\\
\tau(h,xy)=\tau(h_{(1)},y_{(1)})\tau(y_{(2)}\tr h_{(2)},x), \ &
\tau(gh,x)=\tau(g_{(1)},x_{(1)})\tau(h,x_{(2)}\tl g_{(2)}).
\end{split}
\end{align}
\end{definition}

In particular, any skew $H\tb U$-pairing is determined by its values on the generators and is extended using \eqref{eqn:tauskew} via the rule, for $g,h\in H$, $x,y\in U$:
\begin{multline}
%	\begin{split}
\label{eqn:tau-extension}
		\tau(gh,xy)
%&=\tau(g_{(1)}h_{(1)},y_{(1)})\tau(y_{(2)}\tr (g_{(2)}h_{(2)}),x)\\
%&=		\tau(g_{(1)},y_{(1)})\tau(h_{(1)},y_{(2)}\tl g_{(2)})\tau(y_{(3)}\tr (g_{(3)}h_{(2)}),x)\\
%&=\tau(g_{(1)},y_{(1)})\tau(h_{(1)},y_{(2)}\tl g_{(2)})\tau((y_{(3)}\tr g_{(3)})((y_{(4)}\tl g_{(4)})\tr h_{(2)}),x)\\
=\tau(g_{(1)},y_{(1)})\tau(h_{(1)},y_{(2)}\tl g_{(2)})\tau(y_{(3)}\tr g_{(3)},x_{(1)})   \\ \tau((y_{(5)}\tl g_{(5)})\tr h_{(2)},x_{(2)}\tl (y_{(4)}\tr g_{(4)})).
%	\end{split}
\end{multline}

Of course we can interpret the usual skew-pairing as in Section \ref{sec:hhp} as skew $H\otimes U$-pairings. Condition \ref{eqn:tauskew} appears in \cite[Definition 4.1]{bicar}. The definition is motivated by the following lemma.

\begin{lemma}\label{lem:tauskewtwisting}
	Let $(H,U,\tr,\tl)$ be a matched pair of Hopf algebras and let $\tau : H \otimes U \to k$  be a linear map. Consider the map
	\begin{align}\label{eqn:theta-taui}
		\theta_\tau : U \otimes H & \longrightarrow H \otimes U & x \otimes h & \longmapsto \tau(h_{(1)},x_{(1)})\, x_{(2)}\tr h_{(2)}\ot x_{(3)}\tl h_{(3)}.
\end{align}
	Then $\theta_\tau$ is a twisting map if and only if $\tau$ satisfies Condition (\ref{eqn:tauskew}).
\end{lemma}

\begin{proof}
We have, for $x, y \in U$, $h\in H$:
\begin{align*}
\theta_\tau(xy,h) = \tau(h_{(1)},x_{(1)} y_{(1)})\, x_{(2)}y_{(2)}\tr h_{(2)}\ot x_{(3)}y_{(3)}\tl h_{(3)}.
\end{align*}
On the other hand, we have, using the axioms of a matched pair
\begin{align*}
(\id_H\otimes m_U)&\circ (\theta_\tau \otimes \id_U) \circ (\id_U\otimes \theta_\tau)(x\otimes y\otimes h) \\ &= \tau(h_{(1)},y_{(1)})
(\id_H\otimes m_U)\circ (\theta_\tau \otimes \id_U) ( x\otimes y_{(2)}\tr h_{(2)}\ot y_{(3)}\tl h_{(3)}) \\
&= \tau(h_{(1)},y_{(1)}) \tau(y_{(2)} \tr h_{(2)},x_{(1)} )
 x_{(2)}\tr(y_{(3)}\tr h_{(3)})\otimes [x_{(3) } \tl (y_{(4)}\tr h_{(4)})]\cdot y_{(5)}\tl h_{(5)}\\
  &= \tau(h_{(1)},y_{(1)}) \tau(y_{(2)} \tr h_{(2)},x_{(1)} )
 x_{(2)}\tr(y_{(3)}\tr h_{(3)})\otimes x_{(3)} y_{(4)} \tl h_{(4)}\\
 &= \tau(h_{(1)},y_{(1)}) \tau(y_{(2)} \tr h_{(2)},x_{(1)} )
 x_{(2)}y_{(3)}\tr h_{(3)}\otimes x_{(3)} y_{(4)} \tl h_{(4)},
\end{align*}
and this shows that the equivalence of the first  axiom for a twisting map is equivalent to $\tau(h,xy) = \tau(h_{(1)},y_{(1)})\tau(y_{(2)}\tr h_{(2)},x)$. The rest of the proof is similar and straightforward, and is left to the reader.
	\end{proof}

\begin{lemma}\label{lem:tauskewcocycle}
Let $(H,U,\tr,\tl)$ be a matched pair of Hopf algebras and let $\tau : H \otimes U \to k$  be a skew $H\tb U$-pairing. Then
\begin{align*}
\hat{\tau} : H^{} \otimes U  \otimes H^{} \otimes U &\longrightarrow k & g\otimes x \otimes h \otimes y & \longmapsto \tau(h,x) \varepsilon(g)\varepsilon(y)
\end{align*}
is a $2$-cocycle on $H\tb U$, and the algebras $H\#_{\theta_\tau} U$ and $_{\hat{\tau}}\!(H^{}\tb U)$ are isomorphic.
\end{lemma}
	
	\begin{proof}
		First it is clear that $\hat{\tau}$ is convolution invertible since $\tau$ is. In the algebra $H\#_{\theta_\tau} U$, we have
		$$g \# x \cdot h\#y= \tau(h_{(1)},x_{(1)})\, g(x_{(2)}\tr h_{(2)})\# (x_{(3)}\tl h_{(3)})y$$
		while for the product in $_{\hat{\tau}}\!(H^{}\tb U)$ we have
			\begin{align*}
					g\tb x \cdot h\tb y & = \hat{\tau} (g_{(1)} \tb x_{(1)}, h_{(1)}\tb y_{(1)})  g_{(2)}(x_{(2)}\tr h_{(2)})\# (x_{(3)}\tl h_{(3)})y_{(2)}\\
					& = \tau(h_{(1)},x_{(1)})\, g(x_{(2)}\tr h_{(2)})\# (x_{(3)}\tl h_{(3)})y
			\end{align*}
		Hence the two products coincide, which shows simultaneously that $\hat{\tau}$ is a $2$-cocycle on $H\tb U$ and that $H\#_{\theta_\tau} U\simeq\, _{\hat{\tau}}\!(H^{}\tb U)$.
	\end{proof}

\begin{lemma}
Let $(H,U,\tr,\tl)$ be a matched pair of Hopf algebras and let $\tau : H \otimes U \to k$  be a skew $H\tb U$-pairing. Then the maps
$$\tl_\tau\colon U\ot H\to U, \qquad \tr_\tau \colon U\ot H\to H$$
defined by
\[  x \tr_\tau h = \tau(h_{(1)},x_{(1)}) \tau^{-1}(h_{(3)}, x_{(3)}) x_{(2)} \triangleright h_{(2)} , \ x \tl_\tau h = \tau(h_{(1)},x_{(1)}) \tau^{-1}(h_{(3)}, x_{(3)}) x_{(2)} \triangleleft h_{(2)}
\]
define a matched pair  of Hopf algebras $(H,U,\tr_\tau ,\tl_\tau)$. We denote by 
$H\overset{\tau}\tb U$
the corresponding bicrossed product.
\end{lemma}

\begin{proof} %\comj{To be written.} \comy{Written:}
	It is clear that $(U,\triangleleft_\tau)$ is a left $H$-module coalgebra and $(H,\triangleright_\tau)$ is a right $U$-module coalgebra.	Next we check the conditions in \eqref{eqn:compatibilities-matchedpair} for $\tl_\tau$ and $\tr_\tau$; using this condition for $\tl$ and $\tr$, together with \eqref{eqn:tauskew}. We have that, for $x\in U$ and $a,b\in H$, 
	\begin{align*}
&(x_{(1)}\tr_\tau a_{(1)})((x_{(2)}\tl_\tau a_{(2)})\tr_\tau b)=\tau(a_{(1)},x_{(1)}) \tau^{-1}(a_{(3)}, x_{(3)}) 
\tau(a_{(4)},x_{(4)}) \tau^{-1}(a_{(8)}, x_{(8)})\\
&\hspace*{5cm} \tau(b_{(1)},x_{(5)} \triangleleft a_{(5)}) \tau^{-1}(b_{(3)}, x_{(7)} \triangleleft a_{(7)}) (x_{(2)} \triangleright a_{(2)})(x_{(6)} \triangleleft a_{(6)}) \triangleright b_{(2)}\\
&\stackrel{\eqref{eqn:compatibilities-matchedpair}}{=}
\tau(a_{(1)},x_{(1)}) \tau^{-1}(a_{(6)}, x_{(6)})\tau(b_{(1)},x_{(2)} \triangleleft a_{(2)}) \tau^{-1}(b_{(3)}, x_{(5)} \triangleleft a_{(5)}) (x_{(3)} \triangleright a_{(3)})(x_{(4)} \triangleleft a_{(4)}) \triangleright b_{(2)}\\
&\stackrel{\eqref{eqn:tauskew},\eqref{eqn:compatibilities-matchedpair}}{=}\tau(a_{(1)}b_{(1)},x_{(1)}) \tau^{-1}(a_{(3)}b_{(3)}, x_{(3)}) x_{(2)} \triangleright (a_{(2)}b_{(2)})=x\tr_\tau(ab).
	\end{align*}
Similarly one checks that $(yx)\tl a=(y\tl (x_{(1)}\tr a_{(1)}))(x_{(2)}\tl a_{(2)})$ for $y,x\in U$, $a\in H$. Finally, if $y\in U, a\in H$:
	\begin{align*}
			y_{(1)}\tl_\tau &a_{(1)}\ot y_{(2)}\tr_\tau a_{(2)}=y_{(2)}\tl_\tau a_{(2)}\ot y_{(1)}\tr_\tau a_{(1)}
	\end{align*}
as this property holds for $\tl,\tau$ and the cancellation of the proper $\tau$ and $\tau^{-1}$ factors.
	\end{proof}

We conclude with the announced generalization of Proposition \ref{prop:smashgaloishopfhopfskew}.

\begin{proposition}\label{prop:skewbicrossed}
Let $(H,U,\tr,\tl)$ be a matched pair of Hopf algebras and let $\tau : H \otimes U \to k$  be a skew $H\tb U$-pairing. Then the twisted tensor product  $H\#_{\theta_\tau} U$ is a $H\overset{\tau} \tb U$-$H\tb U$-bi-Galois object, whose left and right coactions are induced by the comultiplication of the tensor product coalgebra $H\otimes U$.
\end{proposition}

\begin{proof} 	%\comj{To be written.} \comy{Written:}
As $H\#_{\theta_\tau} U\simeq {_{\hat{\tau}}\!(H^{}\tb U)}$ by Lemma \ref{lem:tauskewcocycle},  then  $H\#_{\theta_\tau} U$ is a $(H \tb U)^{\hat{\tau}}$-$H\tb U$-bi-Galois object, whose left and right coactions are induced by the comultiplication of the tensor product coalgebra $H\otimes U$. Moreover, the algebra $(H \tb U)^{\hat{\tau}}$ is unique with this property, up to isomorphism.

To end the proof, we observe that $H\overset{\tau} \tb U\simeq (H \tb U)^{\hat{\tau}}$, in a similar computation to that of the proof of Lemma \ref{lem:tauskewcocycle}, as the products coincide, see \eqref{eqn:cocycle-def}.
	\end{proof}

\section{Hopf-Galois objects for bicrossed product Hopf algebras}\label{sec:hopf-galois-bicrossed}

%We now generalize  Proposition \ref{prop:smashgaloishopfhopf} to the setting of bicrossed product of Hopf algebras.

The case in which the  Hopf-Galois objects for Hopf algebra bicrossed products are bicleft were already studied in \cite{sc02,bicar}. 
The relevant result there is the Kac exact sequence, which, when the two factors are cocommutative,  contains all the  information about Hopf-Galois objects over the bicrossed product. In this section we discuss Hopf-Galois objects over bicrossed products of Hopf algebras in full generality.

\subsection{General setting}
Recall from \eqref{eqn:twisting-bicrossed} that given a matched pair of Hopf algebras $(H,U,\triangleright, \triangleleft)$,  the associated twisting map $U\ot H\to H\ot U$ is denoted by $\omega_{\tb}$, so $\omega_{\tb}(x,a)=x_{(1)}\tr a_{(1)}\ot x_{(2)}\tl a_{(2)}$.

\begin{lemma}\label{lem:bicrossed}
Let $(H,U,\triangleright, \triangleleft)$ be a matched pair of Hopf algebras and let
 $E=H\tb U$ be the associated Hopf algebra bicrossed product. Let $A$, resp.~$R$, be a right $H$-Galois object, resp.~$U$-Galois object. Let $\theta\colon R\ot A\to A\ot R$ be a twisting map. Then $A\#_\theta R$ is a right $E$-comodule algebra with coaction $\rho^{\#}\coloneqq\rho_{A\ot R}$:
	\begin{align*}
		\rho_{A\ot R}&\colon A\#_\theta R\to A\#_\theta R\ot H\tb U & 
		a\# r & \mapsto a_{(0)}\# r_{(0)}\ot a_{(1)}\tb r_{(1)}
	\end{align*}
if and only if the following diagram commutes:
	\begin{align}\label{eqn:diagram-galois}
\xymatrix{ R\ot A \ar[r]^\theta  \ar[d]^{ \rho_{R\ot A}} & A\ot R \ar[d]^{\rho_{A\ot R}}\\
	R\ot A\ot U\ot H  \ar[r]^{\theta\ot \omega_{\tb}}&  A \ot R \ot H\ot U}
\end{align}
This amounts to the fact that $\theta$ is a morphism of $H\otimes U$-comodules, where the $H\otimes U$-comodule structure on $R\otimes A$ is given by $({\rm id}_{R\otimes A} \otimes \omega_{\tb} )\circ \rho_{R\otimes A}$ (recall that $\omega_{\tb}$ is a coalgebra map) and the one of $A\otimes R$ is given by $\rho^{\#}=\rho_{A\ot R}$.

Explicitly,  $A\#_\theta R$ is a right $E$-comodule algebra if and only if
\begin{align}\label{eqn:theta-comodules}
\theta (r_{(0)}\ot a_{(0)})\ot (r_{(1)}\tr a_{(1)}\tb r_{(2)}\tl a_{(2)})=\rho^\#\left(\theta(r\ot a) \right).
\end{align}
\end{lemma}
\pf
We look for sufficient and necessary conditions so that $\rho^{\#}\colon A\#_\theta R \longrightarrow A\#_\theta R \otimes H\tb U$ 
is an algebra map. 
It is enough to see that 
\[
\rho^{\#}(1\#r)\rho^{\#}(a\#1)=\rho^{\#}\left(\theta (r\ot a)\right), \quad r\in R, \ a\in A,
\]
as $\theta(r,a)=(1\#r)(a\#1)$. It readily follows that:
\begin{align*}
\rho^{\#}(1\#r)\rho^{\#}(a\#1)&=(1 \# r_{(0)} \otimes 1 \tb r_{(1)})(a_{(0)} \# 1 \otimes a_{(1)}  \tb  1)\\
&=\theta(r_{(0)}\ot a_{(0)})\ot r_{(1)}\tr a_{(1)} \tb r_{(2)}\tl a_{(2)}.
\end{align*}	
Thus $\rho^{\#}$ is an algebra map if and only if \eqref{eqn:theta-comodules} holds.
\epf
Using the notation as in \eqref{eqn:theta-explicit}, equation \eqref{eqn:theta-comodules} becomes
\begin{align*}
(r_{(0)}\hr a_{(0)})&\# (r_{(0)}\hl a_{(0)})\ot (r_{(1)}\tr a_{(1)}\tb r_{(2)}\tl a_{(2)})\\
&=(r\hr a)_{(0)}\# (r\hl a)_{(0)}\ot (r\hr a)_{(1)}\tb (r\hl a)_{(1)}.
\end{align*}

We arrive to the description of the right Galois objects over  $E =H \tb U$.

\begin{theorem}\label{thm:bicrossedtwist1}
Let $(H,U,\tr,\tl)$ be a matched pair of Hopf algebras and let $E\coloneqq H\tb U$ be the corresponding bicrossed product. 
\begin{enumerate}
	\item Let $A$, resp.~$R$, be a right $H$-Galois object, resp.~$U$-Galois object, and assume there is a twisting map $\theta\colon R\ot A\to A\ot R$ such that \eqref{eqn:diagram-galois} commutes. Then $A\#_\theta R$ is a right $H\tb U$-Galois object, with the tensor product comodule structure $\rho^\#$ as in Lemma \ref{lem:bicrossed}.
	\item Let $A, A'$, resp.~$R, R'$, be right $H$-Galois objects, resp.~$U$-Galois objects, and assume that there are twisting maps $\theta\colon R\ot A\to A\ot R$ and $\theta'\colon R'\ot A'\to A'\ot R'$such that \eqref{eqn:diagram-galois} commutes. Then the above right $H\tb U$-Galois object  $A\#_\theta R$ and  $A'\#_\theta R'$ are isomorphic if and only if there exist a right $A$-comodule algebra isomorphism $f : A\to A'$ and a right $U$-comodule isomorphism $g : R\to R'$ such that $(f \ot g) \circ \theta = \theta'\circ (g \otimes f)$. 
	\item Conversely, any right $E$-Galois object arises as $A\#_\theta R$ for some $A$, $R$ and $\theta\colon R\ot A\to A\ot R$ as above.
\end{enumerate} 
\end{theorem}
\pf
(1) 
The associated canonical map is as follows:
\begin{align*}
\gamma : A\#_\theta R \otimes A\#_\theta R &\longrightarrow A\#_\theta R \otimes H\tb U\\
a\#x \otimes b\#y & \longmapsto a (x\hr b_{(0)}) \#   (x\hl b_{(0)})y_{(0)} \otimes b_{(1)} \tb y_{(1)}.
\end{align*}
To construct an inverse to $\gamma$, we consider the maps $\kappa_H(h)=h^{(1)} \otimes h^{(2)}\in A\ot A$, $h\in H$, and $\kappa_U(x)=x^{(1)} \otimes x^{(2)}\in R\ot R$, $x\in U$, see $\eqref{eqn:kappa}$, %We set
%\begin{align}
%\begin{split}\label{eqn:inverse-can-bicrossed}
%\kappa_\theta (h\tb x)&=\theta(x^{(1)}\ot h^{(1)})\ot h^{(2)}\# x^{(2)}\\
%&=(x^{(1)}\hr h^{(1)})\# (x^{(1)}\hl h^{(1)})\ot h^{(2)}\# x^{(2)}
%\end{split}
%\end{align}
and define 
\begin{align*}
\gamma' : A\#_\theta R \otimes H\tb U &\longrightarrow A\#_\theta R \otimes A\#_\theta R\\
a\#r \otimes h\tb x & \longmapsto (a\#r)\theta(x^{(1)}\ot h^{(1)})\ot h^{(2)}\# x^{(2)}.
\end{align*}
To see that $\gamma'\circ\gamma=\id$ and $\gamma\circ\gamma'=\id$, it is enough to check that
\begin{align*}
\gamma'\circ\gamma(1\#1\ot b\# y)&=1\#1\ot b\# y, & \gamma\circ\gamma'(1\#1\ot h\tb x)&=1\#1\ot h\tb x.
\end{align*}
We have $\gamma'\circ\gamma=\id$, since:
\begin{align*}
\gamma'\left(\gamma(1\#1\ot b\# y)\right)&=\gamma'( b_{(0)} \#   y_{(0)} \otimes b_{(1)} \tb y_{(1)})\\
&=(b_{(0)} \#   y_{(0)})\theta((y_{(1)})^{(1)}\ot (b_{(1)})^{(1)})\ot (b_{(1)})^{(2)}\# (y_{(1)})^{(2)}\\
&=(b_{(0)} \#   y_{(0)})(1\# (y_{(1)})^{(1)})((b_{(1)})^{(1)}\# 1)\ot (b_{(1)})^{(2)}\# (y_{(1)})^{(2)}\\
&=(b_{(0)} \#   y_{(0)}(y_{(1)})^{(1)})((b_{(1)})^{(1)}\# 1)\ot (b_{(1)})^{(2)}\# (y_{(1)})^{(2)}\\
&\stackrel{\eqref{eqn:canonical-schneider}}{=}(b_{(0)} (b_{(1)})^{(1)}\# 1)\ot (b_{(1)})^{(2)}\# y\stackrel{\eqref{eqn:canonical-schneider}}{=}1\# 1\ot b\# y.
\end{align*}

Similarly, we have $\gamma\circ\gamma'=\id$, since:
\begin{align*}
\gamma&\left(\gamma'(1\#1\ot  h\tb x)\right)=\gamma(\theta(x^{(1)}\ot h^{(1)})\ot h^{(2)}\# x^{(2)})\\
&=\theta(x^{(1)}\ot h^{(1)})( (h^{(2)})_{(0)}\# (x^{(2)})_{(0)})\ot  (h^{(2)})_{(1)}\# (x^{(2)})_{(1)}\\
&=(x^{(1)}\hr h^{(1)}\#x^{(1)}\hl h^{(1)})( (h^{(2)})_{(0)}\# (x^{(2)})_{(0)})\ot  (h^{(2)})_{(1)}\# (x^{(2)})_{(1)}\\
&=((x^{(1)}\hr h^{(1)})  ((x^{(1)}\hl h^{(1)})\hr  (h^{(2)})_{(0)}) \# ((x^{(1)}\hl h^{(1)})\hl  (h^{(2)})_{(0)})  (x^{(2)})_{(0)}) 
\\ &\hspace*{10cm}
\ot  (h^{(2)})_{(1)}\# (x^{(2)})_{(1)}\\
&\stackrel{\eqref{eqn:theta-explicit}}{=}(x^{(1)}\hr h^{(1)}(h^{(2)})_{(0)} )\# (x^{(1)}\hl h^{(1)}(h^{(2)})_{(0)})(x^{(2)})_{(0)} \ot  (h^{(2)})_{(1)}\# (x^{(2)})_{(1)}\\
&\stackrel{\eqref{eqn:kappa-on-elements}}{=}(x^{(1)}\hr 1 )\# (x^{(1)}\hl 1 )(x^{(2)})_{(0)} \ot  (h\# (x^{(2)})_{(1)}\\
&=1\# x^{(1)}(x^{(2)})_{(0)} \ot  h\# (x^{(2)})_{(1)}\stackrel{\eqref{eqn:kappa-on-elements}}{=}1\#1\ot h\tb x.
\end{align*}

(2) An isomorphism of $H\tb U$-Galois objects $A\#_\theta R \to A'\#_{\theta '} R'$ induces an isomorphism between the $H$ and $U$-coinvariant parts, where the respective $H$ and $U$-comodule structures are induced by the coalgebra surjections $\pi_H = {\rm id}_H \otimes U : E\to H$ and $\pi_U = \varepsilon \otimes {\rm id}_U : E \to U$.
Hence our isomorphism has the form $f\otimes g$ for $f$ and $g$ as in the statement. The condition $(f \ot g) \circ \theta = \theta'\circ (g \otimes f)$ comes from the fact that $f\otimes g$ is an algebra map. The converse is an immediate verification.

(3) Follows as  \cite[Lemma 3.1]{sctaft}. We reproduce the details for completeness. Recall that both $H$ and $U$ identify with Hopf subalgebras of $E$, and moreover we have the coalgebra surjections $\pi_H = {\rm id}_H \otimes U : E\to H$ and $\pi_U = \varepsilon \otimes {\rm id}_U : E \to U$ that are respectively left $H$-linear and right $U$-linear. Let $J$ be an  $E$-Galois object and let us set
\begin{align*}
A\coloneqq J^{\co U}=\{z\in J : z_{(0)}\ot z_{(1)}\in F\ot H\}\simeq J\square_E H,\\
R\coloneqq J^{\co H}=\{z\in J : z_{(0)}\ot z_{(1)}\in F\ot U\}\simeq J\square_E U.
\end{align*}
It follows that $A$ is an $H$-Galois object by \cite[Remark 3.11 (2), p.186]{schn}; the same holds for $R$ as an $U$-Galois object. Now, we consider $J$ as a Hopf module in the category $\mathcal{M}^{U}_{R}$, with the $U$-comodule structure induced by $\pi_U$ (this is indeed a Hopf module by the $U$-linearity of $\pi_U$) and thus \cite[Theorem 3.7]{schn} gives an isomorphism (via the multiplication):
\[
A\otimes R = J^{\co U}\ot R\simeq J.
\]
Therefore, there is a twisting map $\theta\colon R\ot A\to A\ot R$ such that $J\simeq A\#_\theta R$ (see the end of Subsection \ref{sub:twist}) and since multiplication is $E$-colinear, the $E$-coaction on $A\otimes R$ has the form as in Lemma \ref{lem:bicrossed}, so that $\theta$ satisfies the conditions there.
\epf

%\begin{remark}
	%	Special type of Hopf-Galois objects for Hopf algebra bicrossed products, the bicleft ones, were already studied in \cite{sc02,bicar}. The relevant result there is the Kac exact sequence, which, when the two factors are cocommutative,  contains all the relevant information about Hopf-Galois objects over the bicrossed product. 
%\end{remark}

\begin{remark}\label{rem:kappa}
Consider the maps $\kappa_H : H \to A\otimes A$ and $\kappa_U : U \to R\otimes R$ as in Subsection \ref{sub:hopfgalois}: 	$\kappa_H(h)=h^{(1)} \otimes h^{(2)}\in A\ot A$, $h\in H$, and $\kappa_U(x)=x^{(1)} \otimes x^{(2)}\in R\ot R$, $x\in U$, see $\eqref{eqn:kappa}$. Then the proof of the above theorem shows that
%	\begin{align}
%	\begin{split}\label{eqn:inverse-can-bicrossed}
%	\kappa_{H\tb U} (h\tb x)&=\theta(x^{(1)}\ot h^{(1)})\ot h^{(2)}\# x^{(2)}\\
%	&=(x^{(1)}\hr h^{(1)})\# (x^{(1)}\hl h^{(1)})\ot h^{(2)}\# x^{(2)}.
%	\end{split}
%	\end{align}
	\begin{align}\label{eqn:inverse-can-bicrossed}
		\kappa_{H\tb U} (h\tb x)&=\theta(x^{(1)}\ot h^{(1)})\ot h^{(2)}\# x^{(2)}=(x^{(1)}\hr h^{(1)})\# (x^{(1)}\hl h^{(1)})\ot h^{(2)}\# x^{(2)}.
\end{align}\end{remark}

%\begin{remark}
%	\comy{In this this Remark you had spotted a  ``}\comj{BIG PROBLEM WITH THESE EQUATIONS, THEY DON'T TAKE THE COCYCLES INTO ACCOUNT}\comy{''. I re-wrote it, taking into account the cocycles in the new version. I changed into a small Subsection \ref{sec:cleft}, because of its length. I also added the division Subsection \ref{subsec:new}, to indicate the end of the computations in the cleft setting. }
%\end{remark}

\subsection{The cleft setting}\label{sec:cleft}
	
Assume that the Galois-objects $A$ and $R$ from Theorem \ref{thm:bicrossedtwist1} are cleft. In particular, there are Hopf 2-cocycles $\sigma\in Z^2(H)$ and $\tau\in Z^2(U)$ so that $A\simeq {}_\sigma H$, $R\simeq {}_\tau U$. Set $E=H\tb U$. Now, any $E$-colinear map $\theta\in \Hom^E (R\ot A, A\ot R)\simeq \Hom^E (U\ot H, H\ot U)$ is determined by a linear map $\psi\colon H\ot U\to k$ in such a way that $\theta=\theta_\psi$, where
\begin{align}\label{eqn:theta-psi}
\theta_\psi(x,h)=\psi(h_{(1)},x_{(1)})\, x_{(2)}\tr h_{(2)}\ot x_{(3)}\tl h_{(3)}, \qquad x\in U, h\in H.
\end{align}
Notice that this also reads: $\theta_\psi(x,h)=\psi(h_{(1)},x_{(1)})\, x_{(2)}\tb h_{(2)}$. As well, recall the notation $\ttl,\ttr$ from Lemma \ref{lem:reverse}.

\begin{lemma}\label{lem:cleft}
Let $\theta\in \Hom^E (R\ot A, A\ot R)\simeq \Hom^E (U\ot H, H\ot U)$ be as in \eqref{eqn:theta-psi}. Then $\theta$ is a twisting map if and only if
\begin{align}
	\begin{split}\label{eqn:twistconditions-psi}
	\tau(x_{(1)},y_{(1)})\psi(h,x_{(2)}y_{(2)})&=\psi(h_{(1)},y_{(1)})\psi(y_{(2)}\tr h_{(2)},x_{(1)})\tau(x_{(2)}\tl (y_{(3)}\tr h_{(3)}),y_{(4)}\tl h_{(4)}),\\
	\sigma(h_{(1)},t_{(1)})\psi(h_{(2)}t_{(2)},x)&=\psi(h_{(1)},x_{(1)})\psi(t_{(1)},x_{(2)}\tl h_{(2)})\sigma(x_{(3)}\tr h_{(3)},(x_{(4)}\tl h_{(4)})\tr t_{(2)}),
	\end{split}
\end{align}
for all $h,k\in H$, $x,y\in U$. 

Assume that $S_H, S_U$ are invertible. If $\psi$ is convolution invertible, then $\theta$ is invertible with 
\begin{align}\label{eqn:theta-invertible}
	\theta^{-1}(h,x)&=\psi^{-1}(h_{(1)}\ttr x_{(1)},h_{(3)}\ttl x_{(3)}) \, h_{(2)}\ttr x_{(2)}\ot h_{(4)}\ttl x_{(4)} \qquad x\in U, h\in H.
\end{align}
\end{lemma}
\pf
The identity in \eqref{eqn:twistconditions} corresponding to $m_R=\tau\ast m_U$, becomes, using \eqref{eqn:compatibilities-matchedpair} with particular emphasis on the fourth identity:
	\begin{align*}
\tau(x_{(1)},y_{(1)})&\psi(h_{(1)},x_{(2)}y_{(2)})x_{(3)}y_{(3)}\tr h_{(2)}\ot x_{(4)}y_{(4)}\tl h_{(3)}\\
&=\psi(h_{(1)},y_{(1)})\psi(y_{(2)}\tr h_{(2)},x_{(1)})\tau(x_{(3)}\tl (y_{(4)}\tr h_{(4)}),y_{(6)}\tl h_{(6)})\\
&\qquad
x_{(2)}y_{(3)}\tr h_{(3)}\ot (x_{(4)}\tl (y_{(5)}\tr h_{(5)}))(y_{(7)}\tl h_{(7)}),	\\
&\stackrel{\eqref{eqn:compatibilities-matchedpair}}{=}\psi(h_{(1)},y_{(1)})\psi(y_{(2)}\tr h_{(2)},x_{(1)})\tau(x_{(3)}\tl (y_{(4)}\tr h_{(4)}),y_{(5)}\tl h_{(5)})\\
&\qquad
x_{(2)}y_{(3)}\tr h_{(3)}\ot (x_{(4)}\tl (y_{(6)}\tr h_{(6)}))(y_{(7)}\tl h_{(7)}),	\\
&=\psi(h_{(1)},y_{(1)})\psi(y_{(2)}\tr h_{(2)},x_{(1)})\tau(x_{(3)}\tl (y_{(4)}\tr h_{(4)}),y_{(5)}\tl h_{(5)})\\
&\qquad
x_{(2)}y_{(3)}\tr h_{(3)}\ot (x_{(4)}y_{(6)})\tr h_{(6)},	\\
&\stackrel{\eqref{eqn:compatibilities-matchedpair}}{=}\psi(h_{(1)},y_{(1)})\psi(y_{(2)}\tr h_{(2)},x_{(1)})\tau(x_{(2)}\tl (y_{(4)}\tr h_{(4)}),y_{(3)}\tl h_{(3)})\\
&\qquad
x_{(3)}y_{(5)}\tr h_{(5)}\ot (x_{(4)}y_{(6)})\tr h_{(6)}.
	\end{align*}
Similarly, the identity in \eqref{eqn:twistconditions} corresponding to $m_A=\sigma\ast m_H$ is:
	\begin{align*}
\sigma(h_{(1)},t_{(1)})&\psi(h_{(2)}t_{(2)},x_{(1)})x_{(2)}\tr h_{(3)} t_{(3)}\ot x_{(3)}\tl h_{(4)}t_{(4)}\\
&=\psi(h_{(1)},x_{(1)})\psi(t_{(1)},x_{(4)}\tl h_{(4)})\sigma(x_{(2)}\tr h_{(2)},(x_{(5)}\tl h_{(5)})\tr t_{(2)})\\
&\qquad (x_{(3)}\tr h_{(3)})((x_{(6)}\tl h_{(6)})\tr t_{(3)})\ot (x_{(7)}\tl  h_{(7)})\tl t_{(4)}\\
&\stackrel{\eqref{eqn:compatibilities-matchedpair}}{=}
\psi(h_{(1)},x_{(1)})\psi(t_{(1)},x_{(3)}\tl h_{(3)})\sigma(x_{(2)}\tr h_{(2)},(x_{(4)}\tl h_{(4)})\tr t_{(2)})\\
&\qquad x_{(5)}\tr h_{(5)}t_{(3)}\ot x_{(6)}\tl h_{(6)}t_{(4)}.
	\end{align*}
Next we check \eqref{eqn:theta-invertible}. Notice that, if we write $h\ot x=y\tb t$, $y\in U, t\in H$, then this is $\theta^{-1}(h,x)=\gamma(y, t)$, for 
\[
\gamma(y, t)=\psi^{-1}(t_{(1)},y_{(1)})y_{(2)}\ttb t_{(2)}.
\] 
Indeed, if we identify $U\ot H=H\ttb H$, then we have that
\begin{align*}
	\gamma\theta(x,h)&=\psi(h_{(1)}, x_{(1)})\gamma(x_{(2)}\tb h_{(2)})=\psi(h_{(1)}, x_{(1)})\psi^{-1}(h_{(2)},x_{(2)})x_{(3)}\ttb h_{(3)}=x\ot h,\\
	\theta\gamma(h,x)&=\theta\gamma(y,t)=\psi^{-1}(t_{(1)}, y_{(1)})\theta(y_{(2)}\ttb t_{(2)})=\psi^{-1}(t_{(1)}, y_{(1)})\psi(t_{(2)},y_{(2)}) y_{(3)}\tb t_{(3)}\\
	&=y\tb t=h\ot x.
\end{align*}
This ends the proof.
\epf

By \eqref{eqn:twistconditions-psi}, if we assume that the values of $\sigma,\tau$ are known, then it is enough to define $\psi$ on the generators of $U\ot H$ and extend it via: 
\begin{align*}
	\psi(ht,xy)=\sigma^{-1}&(h_{(1)},t_{(1)})\tau^{-1}(x_{(1)},y_{(1)})\psi(h_{(2)},y_{(2)})\psi(y_{(3)}\tr h_{(3)},x_{(2)})\\
	%&\tau(x_{(3)}\tl(y_{(4)}\tr h_{(4)}),y_{(6)}\tl h_{(6)})\tau^{-1}(x_{(4)}\tl (y_{(5)}\tr h_{(5)}),y_{(7)}\tl h_{(7)})\\
	&\psi(t_{(2)},y_{(6)}\tl h_{(6)})\psi((y_{(7)}\tl h_{(7)})\tr t_{(3)},x_{(3)}\tl(y_{(4)}\tr h_{(4)}))\\
	&\quad \tau((x_{(4)}\tl (y_{(5)}\tr h_{(5)}))\tl((y_{(8)}\tl h_{(8)})\tr t_{(4)}),y_{(9)}\tl (h_{(9)}t_{(5)}))\\
	&\quad \quad \sigma((x_{(5)}y_{(10)})\tr h_{(10)},((x_{(6)}y_{(11)})\tl h_{(11)})\tr t_{(6)}).
\end{align*}
In general, the determination of the values of a Hopf cocycle $\sigma$ is a hard work; nevertheless there instances in which task has been performed and characterized, see e.g.~\cite{gam,gs}.

%\comy{I added the next remark, which is a simplification of the lemma, but it works for the coming example.}

We the conclude the subsection with a useful remark.

\begin{remark}\label{rem:psi-invertible}
	When one of the cocycles $\sigma,\tau$ is trivial, then \eqref{eqn:twistconditions-psi} become simpler and we can deduce further properties of $\psi$:
	\begin{enumerate}[leftmargin=*]
\item[(a)] Assume that $\tau=\eps$. If the antipode $S=S_U$ is invertible, then the first equation in \eqref{eqn:twistconditions-psi} is
\[
\psi(h,xy)=\psi(h_{(1)},y_{(1)})\psi(y_{(2)}\tr h_{(2)},x)
\]
and then $\psi$ is convolution-invertible, with 
\[
\psi^{-1}(h,x)=\psi(x_{(1)}\tr h,S^{-1}(x_{(2)})).
\]
Indeed, let $\varphi(h,x)=\psi(x_{(1)}\tr h,S^{-1}(x_{(2)}))$. Then
\begin{align*}
	\psi\ast\varphi(h,x)=\psi(h_{(1)},x_{(1)})\psi(x_{(2)}\tr h_{(2)},S^{-1}(x_{(3)}))=\psi(h,S^{-1}(x_{(2)})x_{(1)})=\eps(h)\eps(x).
\end{align*}
On the other hand, letting $y=S^{-1}(x)$: 
\begin{align*}
	\varphi\ast\psi(h,x)&=\psi(x_{(1)}\tr h_{(1)},S^{-1}(x_{(2)}))\psi(h_{(2)},x_{(3)})=\psi(S(y_{(3)})\tr h_{(1)},y_{(2)})\psi(h_{(2)},S(y_{(1)}))\\
	&=\psi(S(y_{(5)})\tr h_{(1)},y_{(2)})\psi(y_{(3)}\tr (S(y_{(4)})\tr h_{(2)}),S(y_{(1)}))\\
	&=\psi((S(y_{(3)})\tr h)_{(1)},(y_{(2)})_{(1)})\psi((y_{(2)})_{(2)}\tr (S(y_{(3)})\tr h)_{(2)},S(y_{(1)}))\\
	&=\psi(S(y_{(3)})\tr h,S(y_{(1)})y_{(2)})=\eps(h)\eps(y)=\eps(h)\eps(x).
\end{align*}

\item[(b)] On the other hand, if $\sigma=\eps$ and $S_H$ is invertible, then $\psi$ is invertible. The second equation in \eqref{eqn:twistconditions-psi} becomes
\[
\psi(hk,x)=\psi(h_{(1)},x_{(1)})\psi(k,x_{(2)}\tl h_{(2)})
\]
and we get that:
\[
\psi^{-1}(h,x)=\psi(S_H(h_{(2)}),x\tl h_{(1)}).
\]
In this setting, the inverse $S^{-1}$ of $S=S_H$ is used in the verification, arguing as in the previous case. Indeed, if we now set $\varphi(h,x)=\psi(S(h_{(2)}),x\tl h_{(1)})$ it is straightforward to check that $\psi\ast\varphi=\eps$, while, for $k=S^{-1}(h)$:
\begin{align*}
	\varphi\ast\psi(h,x)&=\psi(S(h_{(2)}),x_{(1)}\tl h_{(1)})\phi(h_{(3)},x_{(2)})=\psi(k_{(2)},x_{(1)}\tl S^{-1}(k_{(3)}))\psi(S^{-1}(k_{(1)}),x_{(2)})\\
	&=\psi(k_{(2)},x_{(1)}\tl S^{-1}(k_{(5)})))\psi(S^{-1}(k_{(1)}),x_{(2)}\tl (S^{-1}(k_{(4)}) y_{(3)}))\\
	&=\psi(k_{(2)},(x\tl S^{-1}(k_{(4)}))_{(1)}))\psi(S^{-1}(k_{(1)}),(x\tl S^{-1}(k_{(4)}))_{(2)}\tl y_{(3)})\\
	&=\psi(k_{(2)}S^{-1}(k_{(1)}),x\tl S^{-1}(y_{(3)}))=\eps(k)\eps(x)=\eps(h)\eps(x).
\end{align*} 
	\end{enumerate}
\end{remark}

%Next we produce an example for Lemma \ref{lem:cleft}, in the lines of \S\ref{subsec:examples-skew-pairings}.

\subsection{The left bicrossed product Hopf algebra}\label{subsec:new}
We recall the notation and results on cogroupoids from \S \ref{sec:cogroupoids}; in particular the subcogroupoid $\C_H\subset \C$ from Lemma \ref{lem:subco} associated to a cogroupoid $\C$, an object $X\in\ob\C$ and a sub Hopf algebra $H\subseteq \C(X,X)$.

\begin{theorem}\label{thm:bicrossedtwist2}
Let $(H,U,\tr,\tl)$ be a matched pair of Hopf algebras. %, and $\omega_{\tb}\colon U\ot H\to H\ot U$ be as above. 

	\begin{enumerate}
		\item Let $L, Q$ be Hopf algebras and  let $A$, resp.~$R$, be an $(L,H)$-Galois object, resp.~$(Q,U)$-Galois object. Assume there is a twisting map $\theta\colon R\ot A\to A\ot R$ such that \eqref{eqn:diagram-galois} commutes. Then 
		the associated left Hopf algebra $F=L(A\#_\theta R,H\tb U)$ is a bicrossed product $L\btb Q$ for some actions $\btr \colon Q\ot L\to L$ and $\btl \colon Q\ot L\to Q$ making the following diagram commute:
			\begin{align}\label{eqn:diagram-galois-left}
		\xymatrix{ R\ot A \ar[r]^\theta  \ar[d]^{ \lambda_{R\ot A}} & A\ot R \ar[d]^{\lambda_{A\ot R}}\\
			Q\ot L\ot  R\ot A  \ar[r]^{\ \omega_{\btb}\ot\theta}&  L\ot Q\ot A \ot R}
		\end{align}	
		where $\lambda_{A\ot R}$ and $\lambda_{R\ot A}$ denote the diagonal left coactions induced by those of $A$ and $R$.
		\item Let $F$ be a Hopf algebra such that there exists an $(F,H\tb U)$-bi-Galois object. Then $F\simeq L\btb Q$ for some Hopf algebras $L$, $Q$ as above.
	\end{enumerate} 
\end{theorem}
\pf
(1) We use \cite[Theorem 2.11]{bic} to consider a cogroupoid $\C$ with two objects $X,Y$ such that 
\begin{align*}
\C(X,X)&=H\tb U, & \C(Y,X)&=A\#_\theta R,
\end{align*}
so that $H,U$ become Hopf subalgebras of  $\C(X,X)$ and the we have a bijective multiplication 
	\begin{align}\label{eqn:mult-bij}
	m\colon H\ot U\to \C(X,X).
\end{align} 
 Consider the subcogroupoids $\C_H,\C_U\subseteq \C$ as in Lemma \ref{lem:subco}; we have $\C_H(Y,X)\simeq A$, $\C_U(Y,X)\simeq R$, and $\C_H(Y,Y)\simeq L$, $\C_U(Y,Y)\simeq Q$ by uniqueness of left Hopf algebras in \cite[Theorem 3.5]{sc1}. 

We need to show that the multiplication map $m:\C_H(Y,Y)\ot \C_U(Y,Y)\to \C(Y,Y)$ is bijective. We will do so by fitting it into the following commutative diagram with bijective arrows, where $\square=\square_{\C(X,X)}$ and the bijectivity of each arrow corresponds to the label on top:
	\begin{align*}
\xymatrix{  \C_H(Y,Y)\ot \C_U(Y,Y)\ar@{<->}[d]_{\eqref{eqn:cotensor}}\ar[r]^m & \C(Y,Y)\ar@{<->}[dd]^{\eqref{eqn:bicom-iso}}\\
{\begin{matrix}
	\left(\C(Y,X)\square H\square\,\C(X,Y) \right) \ot \left(\C(Y,X)\square U\square\,\C(X,Y)\right)
	\end{matrix}}
\ar@{<->}[d]_{\eqref{eqn:monoidal}} && \\
{\begin{matrix}
\C(Y,X)\square (H\square\,\C(X,Y)) \ot (U\square\,\C(X,Y)) )
	\end{matrix}}
\ar@{<->}[dr]_{\eqref{eqn:monoidal}}&\C(Y,X)\square\,\C(X,X) \square\,\C(X,Y) \ar@{<->}[d]^{\eqref{eqn:mult-bij}}\\
&\C(Y,X)\square\left( (H\ot U)\square\,\C(X,Y) \right). & }
\end{align*}
This shows that $F$ is a bicrossed product $L\btb U$ of $L\simeq\C_H(Y,Y)$ and $Q\simeq\C_U(Y,Y)$. 

The commutation of \eqref{eqn:diagram-galois-left} will follow from the obvious left version of that of \eqref{eqn:diagram-galois} in Lemma \ref{lem:bicrossed}, provided  we show that the left $L\btb Q$-coaction on $A\#_\theta$ R has the appropriate form. We identify $L\simeq L(A,H)=(A\ot A^{\op})^{\co H}$ and $Q\simeq L(R,U)=(R\ot R^{\op})^{\co U}$, so that 
\[
\lambda_A(a)=a_{(0)}\ot \kappa_H(a_{(1)})=a_{(0)}\ot (a_{(1)})^{(1)}\ot (a_{(1)})^{(2)}
\]
and similarly $\lambda_R(r)=r_{(0)}\ot (r_{(1)})^{(1)}\ot (r_{(1)})^{(2)}$. 

In the same spirit, we identify $L\btb Q\simeq L(A\#_\theta R,H\tb U)=(A\#_\theta R\ot (A\#_\theta R)^{\op})^{\co H\tb U}$, and we have
\begin{align*}
\lambda_{A\#_\theta R}(a\# r)&=a_{(0)}\# r_{(0)}\ot \kappa_{H\tb U} (a_{(1)}\tb r_{(1)})\\
&\stackrel{\eqref{eqn:inverse-can-bicrossed}}{=}a_{(0)}\# r_{(0)}\ot \theta((r_{(1)})^{(1)}\ot (a_{(1)})^{(1)})\ot (a_{(1)})^{(2)}\# (r_{(1)})^{(2)}\\
&[(a_{(0)}\# 1\ot (a_{(1)})^{(1)}\#1)\ot (a_{(1)})^{(2)}\# 1][(1\# r_{(0)}\ot 1\# (r_{(1)})^{(1)})\ot 1\# (r_{(1)})^{(2)}]
%\\ &=\lambda_{A\ot R}(a\#r)
\end{align*}
where the last product between the terms in square brackets $[,]$  is taken in the algebra
\[
A\#_\theta R\ot (A\#_\theta R)^{\op}\ot A\#_\theta R.
\]
This shows that the following diagram commutes:
\begin{align*}
\xymatrix{A \#_\theta R \ar[r]^-{\lambda_{A\#_\theta R}}   \ar[d]_{\lambda_{A\otimes R}} &  (A\#_\theta R\ot (A\#_\theta R)^{\op})^{\co H\tb U} \otimes A \#_\theta R \\
	(A\ot A^{\op})^{\co H} \otimes (R\ot R^{\op})^{\co U}\otimes A \otimes R. \ar[ur]_{\xi \otimes \id_{A\otimes R}} 
}
\end{align*}
Here $\xi$ is, as at the end of Subsection \ref{sub:twist}, the composition of the multiplication in $A\#_\theta R\ot (A\#_\theta R)^{\op}$ together with the inclusion map
$$\iota_A \otimes \iota_R : (A\ot A^{\op})^{\co H} \otimes (R\ot R^{\op})^{\co U} \to (A\#_\theta R\ot (A\#_\theta R)^{\op})^{\co H\tb U} \otimes (A\#_\theta R\ot (A\#_\theta R)^{\op})^{\co H\tb U}$$
where $\iota_A$ and $\iota_R$ are the canonical inclusions
\begin{align*}
\iota_A : (A\ot A^{\op})^{\co H} &\longrightarrow (A\#_\theta R\ot (A\#_\theta R)^{\op})^{\co H\tb U} & a \otimes b &\longmapsto a\#1 \otimes b\#1,\\ 
\iota_R : (R\ot R^{\op})^{\co U} &\longrightarrow (A\#_\theta R\ot (A\#_\theta R)^{\op})^{\co H\tb U} & r \otimes s &\longmapsto 1\#r \otimes 1\#s.
\end{align*}
Since we know that $L\btb Q\simeq (A\#_\theta R\ot (A\#_\theta R)^{\op})^{\co H\tb U}$, the map $\xi$ is a Hopf algebra isomorphism, and it follows in particular that $\lambda_{A\otimes R}$ is an algebra map, so the left analogue of Lemma \ref{lem:bicrossed} applies, and the diagram (\ref{eqn:diagram-galois-left}) commutes.

(2) This follows by the uniqueness of the left Hopf algebra associated to a right Galois object and the previous results. Indeed, if $J$ is a $(F,H\tb U)$-bi-Galois object, then $J=A\#_\theta R$ for some right Galois objects $A$ and $R$ over $H$ and $U$ respectively by Theorem \ref{thm:bicrossedtwist1}. %As well, $F\simeq L(J,H\tb U)$. 
If $L=L(A,H)$ and $Q=L(R,U)$ are the corresponding left Hopf algebras so that $A$ is $(L,H)$-bi-Galois and $R$ is $(Q,U)$-bi-Galois, then $J$ is $(L\btb Q,H\tb U)$-bi-Galois by (1) and thus $L\btb Q\simeq L(J,H\tb U) \simeq F$. 
	\epf

\begin{corollary}
With the notation of Theorem \ref{thm:bicrossedtwist2}, assume that $A$ and $R$ are cleft. 

Let $\psi\colon H\ot U\to k$ be such that $\theta=\theta_\psi$ as in \eqref{eqn:theta-psi}. If $\psi$ is invertible, then
\[
\omega_{\btb}(x,h)=\psi(h_{(1)},x_{(1)})\psi^{-1}(h_{(4)},x_{(4)})x_{(2)}\tr h_{(2)}\ot x_{(3)}\tl h_{(3)}.
\]
\end{corollary}
\pf
In this setting, $R\simeq U$ as $U$-comodules and $A\simeq H$ as $H$-comodules, so we may assume, without loss of generality, that $R=U$ and $A=H$ in \eqref{eqn:diagram-galois-left}. By following the arrows, we get
\begin{align*}
\omega_{\btb}(x_{(1)},h_{(1)}) &\ot \psi(h_{(2)},x_{(2)}) x_{(3)}\tr h_{(3)}\ot x_{(4)}\tl h_{(4)}  
\\
&=\psi(h_{(1)},x_{(1)}) x_{(2)}\tr h_{(2)}\ot x_{(3)}\tl h_{(3)}\ot x_{(4)}\tr h_{'(4)}\ot x_{(5)}\tl h_{(5)}.
\end{align*}
We apply $\id_H\ot\,\id_U\ot \eps_H\ot\eps_U$ and obtain
\begin{align*}
\psi(h_{(2)},x_{(2)})\omega_{\btb}(x_{(1)},h_{(1)})=\psi(h_{(1)},x_{(1)}) x_{(2)}\tr h_{(2)}\ot x_{(3)}\tl h_{(3)}
\end{align*}
from where the result follows.
\epf

\begin{remark}
	In the case   of a tensor product $E=H\ot U$,   \cite[Proposition 3.7]{sctaft} describes the Hopf-Galois objects  with $\theta$ arising from a skew pairing $\tau\colon Q\ot L\to k$, so that $L\btb Q\simeq L\tb_\tau Q$. In the general situation of Theorem \ref{thm:bicrossedtwist1} and Theorem \ref{thm:bicrossedtwist2}, we do not see how to reach such a simple and elegant statement as \cite[Proposition 3.7]{sctaft}, since it seems to us that there is no canonical bicrossed product $L\btb Q$ to start with.
	
\end{remark}

%\end{example}

%\comj{Move example from the end and translate it into current language.}

%\coma{We end this section with another well-known example; which we describe in our framework.}

\section{A complete example}\label{sec:exbicrossed}

In this section we illustrate the previous results with an example  of a bicrossed product that does not fall into the classes of bicrossed products mentioned in Example  \ref{ex:basicbicrossed}.

%\subsection{A complete example}\label{sec:exbicrossed}
%In this subsection we present and study in detail an example of a bicrossed product Hopf algebra that does note fit into the classes presented in Example \ref{ex:basicbicrossed}.

\subsection{The Hopf algebra $E=H\tb U$}\label{sec:E}
Throughout the section, let $U, H$ be the graded pointed Hopf algebras given by the smash products (also called {\it bosonizations} in this context) $H=k[b,c]\# k\Z$, $\Z=\lg h\rg$ and $U=k[a]\# k\Z$, $\Z=\lg g\rg$. Here $k[a]$, $k[b,c]$ are polynomial algebras.
Each smash product is defined so that $hb=-bh$, $h c=-ch$, $ga=-ag$. 

The comultiplication is such that $g\in U, h\in H$ are group-like elements and
\begin{align}\label{eqn:comult}
	\Delta(a)&=a\ot1+g^2\ot a, & \Delta(b)&=b\ot1+h^2\ot b, & \Delta(c)&=c\ot1+h^2\ot c.
\end{align}

\begin{definition}\label{def:E}
Let  $E$ be the smash product $k[a,b,c]\# k\Z^{2}$, i.e.~the $k$-algebra generated by $a,b,c,h^{\pm1},g^{\pm1}$ with $gh=hg$ and commutation:
\begin{align}\label{eqn:relations-basic}
	ha=-ah, \quad h b=-bh, \quad h c=-ch, \qquad  ga=-ag, \quad g b=-cg, \quad g c=-bg,
\end{align}
together with the quadratic identities
\begin{align}\label{eqn:relations-quad}
	ab-ba&=0, & ac-ca&=0, & bc-cb&=0.
\end{align}
\end{definition}

It is easy to see that $E$ is a Hopf algebra with coalgebra structure given by that of $H\ot U$. Moreover, it also follows that $E$ factors through $H$ and $U$ via the canonical injections $\iota\colon H\hookrightarrow E$ and $j:U\hookrightarrow E$. Hence $E\simeq H\tb U$.

Next remark shows that the bicrossed product  $E\simeq H\tb U$  above is neither a smash product (both actions are nontrivial) nor it comes from a pairing between $H$ and $U$.
It also establishes some properties of the actions involved, that will become useful further on.

\begin{remark}\label{rem:actions}
	With the notation we make the following observations:
	\begin{enumerate}
		\item[(a)] 
		Both actions $\tl\colon U\ot H\to U$ and $\tr\colon U\ot H\to H$ are non-trivial as
		\[
		ah=-ha\Rightarrow a\tl h=-a, \qquad gb=-cg\Rightarrow g\tr b=-c. 
		\]
		\item[(b)] The bicrossed product $H\tb U$ does not come from a pairing. Indeed, for any pairing $\tau\colon U\ot H\to k$, we get, by \eqref{eqn:cocycle-def}:
		\begin{align*}
			g.b&= \tau(g,b_{(1)})b_{(2)}g\tau^{-1}(g,b_{(3)})=\tau(g,b)g+\tau(g,h^2)bg+\tau(g,h^2)h^2g\tau^{-1}(g,b).
		\end{align*}
		Hence $g.b \in k\{g,bg,h^2g\}$ while in $H\tb U$ we have $gb=-cg\notin k\{g,bg,h^2g\}.
		$
		\item[(c)] Observe  that $a\tr\_=0$ and $\_\tl b=0$, same for $c$. We write down the tables with the actions $\tr$ and $\tl$ in terms of the generators:
		\begin{align*}
			&
			\begin{tabular}{|c|c|c|c|}
				\hline
				$\tr$ & $b$ & $c$ & $h$ 	 \\
				\hline
				$a$ & $0$ & $0$ & $0$ 	 \\
				\hline
				$g$ & $-c$ & $-b$ & $h$ 	 \\
				\hline 	
			\end{tabular} 
			&
			\begin{tabular}{|c|c|c|c|}
				\hline
				$\tl$ & $b$ & $c$ & $h$ 	 \\
				\hline
				$a$ & $0$ & $0$ & $-a$ 	 \\
				\hline
				$g$ & $0$ & $0$ & $g$ 	 \\
				\hline 	
			\end{tabular}
		\end{align*}
		Furthermore, notice that both $H=\sum_{i\geq 0}H_i$ and $U=\sum_{j\geq 0}U_j$ are coradically graded Hopf algebras;  here $H_i=k[b^nc^m:n+m=i]\# k\Z$ and $U_j=k[a^j]\# k\Z$. 
		It is thus easy to check that $a\tr H=0$ and $U\tl b=U\tl b=0$, using \eqref{eqn:compatibilities-matchedpair}. Moreover, this extends to
		\begin{align}\label{eqn:action-zero}
			U_{>0}\tr H&=0, & U\tl H_{>0}&=0.
		\end{align}
	\end{enumerate}
\end{remark}

We now introduce some notation that will become handy to deal with examples.
\begin{notation}
	Let $H$, and $U$ as above and consider the decomposition $H=\sum_{i\geq 0}H_i$ and $U=\sum_{j\geq 0}U_j$ in Remark \ref{rem:actions} (c). 
	We shall write, for $f\in H_n$, $u\in U_n$:
	\begin{align}
		\label{eqn:Delta-rf}
		\Delta(f)=f\ot r(f)+\ell(f)\ot f+f_{\underline{(1)}}\ot f_{\underline{(2)}}, \qquad \Delta(u)=u\ot r(u)+\ell(u)\ot u+u_{\underline{(1)}}\ot u_{\underline{(2)}}
	\end{align}
	with $r(f),\ell(f)\in\lg h\rg$ such that 
	\[
	f_{\underline{(1)}}\ot f_{\underline{(2)}}=\Delta(f)-f\ot r(f)+\ell(f)\ot f\in \sum_{i=1}^{n-1}H_i\ot H_{n-i}.
	\] 
	Similarly,
	$r(u),\ell(u)\in\lg g\rg$ and $u_{\underline{(1)}}\ot u_{\underline{(2)}}\in \sum_{i=1}^{n-1}U_i\ot U_{n-i}$. 
\end{notation}

%%%%%%%%%%%%%%%%%%%%%%%%%
%%%%%%%%%%%%%%%%%%%%%%%%

%\begin{example}
%	\comj{Write down the pertinent things from Example \ref{exa:bicrossed}.}
%	\comy{Made it a Subsubsection, because of length.}
%\end{example}

%\subsubsection{Examples of skew-pairings}\label{subsec:examples-skew-pairings}

\subsection{Deformations} In this part we introduce two families of algebras, obtained by deforming the relations of the Hopf algebra $E$. 
\begin{definition}\label{def:Edeformations}
Let us fix $\alpha,\beta\in k$ and $\lambda\in k^\times$ scalars with the following restrictions:
\begin{align}\label{eqn:conditions-abl}
\alpha(1+\lambda^4)&=0, & \beta(1-\lambda^2)&=0.
\end{align}
\begin{enumerate}[leftmargin=*]
\item[(a)] Let  $E_{\alpha,\beta}^\lambda$ be the $k$-algebra generated by $a,b,c,h^{\pm1},g^{\pm1}$ with relations $gh=hg$ and:
\begin{align}\label{eqn:relations-basic-def}
	ha=-\lambda^2 ah, \quad h b=-bh, \quad h c=-ch, \qquad  ga=-ag, \quad g b=-\lambda^2 cg, \quad g c=-\lambda^2 bg,
\end{align}
together with the identities
\begin{align*}
	ab-\lambda^4ba&=\beta(1-g^2h^2), & ac-\lambda^4ca&=\beta(1-g^2h^2), & bc-cb&=\alpha(1-h^4).
\end{align*}
\item[(b)] Let  $A_{\alpha,\beta}^\lambda$ be the $k$-algebra generated by $a,b,c,h^{\pm1},g^{\pm1}$ with relations $gh=\lambda hg$ and \eqref{eqn:relations-basic-def}, together with the identities
\begin{align}\label{eqn:relations-quad-def}
	ab-\lambda^4ba&=\beta, & ac-\lambda^4ca&=\beta, & bc-cb&=\alpha.
\end{align}
\end{enumerate}
\end{definition} 

We remark that \eqref{eqn:conditions-abl} above forces $\alpha\beta=0$; moreover $\alpha=\beta=0$ whenever $\lambda\neq\pm1$ or $\lambda^4\neq -1$. 

\medspace

It is easy to see that $E_{\alpha,\beta}^\lambda$ is a Hopf algebra with comultiplication \eqref{eqn:comult}. 
Moreover, it can be checked that the algebras $A_{\alpha,\beta}^\lambda$ are $(E,E_{\alpha,\beta}^\lambda)$-biGalois objects. 
This is indeed the content of this part, as the result of applying the ideas in Sections \ref{sec:twistings} and \ref{sec:hopf-galois-bicrossed}.

We refer to Corollary \ref{cor:summarize} for a general picture summarizing our results.

\subsection{Skew $H\tb U$-pairings}\label{sub:exskew}
We now discuss the skew $H\tb U$-pairings for the matched pair of Hopf algebras $(H,U,\tl,\tr)$ in \S\ref{sec:E}.
We construct two such maps in Examples \ref{exa:tau0} and \ref{exa:tau1} and then show that these exhaust all the skew $H\tb U$-pairings in Proposition \ref{prop:all-skew}. Along the way we describe the corresponding Galois objects and left Hopf algebras.
%$H\coloneqq k[b,c]\# k\lg h\rg$ and  $U=k[a]\# k\lg g\rg$,  from Example \ref{exa:bicrossed}. We recall the decompositions $H=\sum_{i\geq 0}H_i$ and $U=\sum_{j\geq 0}U_j$ from Remark \ref{rem:actions} as well as the identities \eqref{eqn:action-zero} and the notation in 
%\eqref{eqn:Delta-rf}. 

We observe that in this setting, using \eqref{eqn:action-zero}, conditions \eqref{eqn:tauskew} become:
\begin{align}
	\label{eqn:tauskew-restricted}
	\tau(f,uu')&=\tau(f_{(1)},u')\tau(r(u')\tr f_{(2)},u), 
	\ &
	\tau(ff',u)&=\tau(f,u_{(1)})\tau(f',u_{(2)}\tl r(f)),
\end{align}for any $f\in H_s$, $f\in H_t$, $u\in U_n$, $u'\in U_m$, $s,t,n,m\geq 0$.

\medskip

We begin with a quick remark.

\begin{remark}\label{rem:quick}
	\begin{enumerate}[leftmargin=*]
		\item Any linear map $\tau\colon H\ot U\to k$ is convolution invertible if and only if it is invertible in the coradical $E_0=H_0\ot U_0\simeq k\lg g,h\rg$ of $E$: that is if and only if $\lambda\coloneqq\tau(g,h)\neq 0$. 
		\item Conditions \eqref{eqn:tauskew}  restricted to $H_0\ot U_0$ are equivalent to stating that  $\tau^h=\tau(h,-)\colon \lg g\rg\to k$ and $\tau^g=\tau(-,g)\colon \lg h\rg\to k$ are algebra maps and thus they are determined by $\lambda$.
	\end{enumerate}
\end{remark}

Our first example is rather straightforward and it does not lead to interesting consequences. 
We include it for completeness, see Proposition \ref{prop:all-skew} below.

\begin{example}\label{exa:tau0}
	For each $\lambda\neq 0$, set $\tau=\tau_{\lambda}\colon H\ot U\to k$ be the linear map so that
	\begin{align*}
		\tau(h^p,g^q)&=\lambda^{pq}, & \tau_{|H_i\ot U_j}&=0, \ i+j>0.
	\end{align*}
	Then $\tau$ is skew $H\tb U$-pairing. As well, $H\overset{\tau} \tb U\simeq E_{0,0}^\lambda$ and $H\#_{\theta_\tau} U\simeq A_{0,0}^\lambda$.
	
\pf	First, Remark \ref{rem:quick} shows that $\tau$ is convolution invertible and identities \eqref{eqn:tauskew} hold on $H_0\ot U_0$. Next, if either $f\in H_{>0}$ $f'\in H_{>0}$ or $u\in U_{>0}$, we have that $\tau(ff',u)=0$. On the other hand, the left hand side of the corresponding equation in \eqref{eqn:tauskew} is $\tau(f,u_{(1)})\tau(f',u_{(2)}\tl r(f))$ by \eqref{eqn:tauskew-restricted}. Again, if $f\in H_{>0}$, $f'\in H_{>0}$ or $t\in U_{>0}$, then this is also zero by definition. The other identity in \eqref{eqn:tauskew} follows similarly.
	
	Notice that the corresponding cocycle $\hat{\tau}$ is concentrated in the coradical $\Z^2$. We thus get that  the deformation $H\overset{\tau} \tb U\simeq (H\tb U)^{\hat{\tau}}$ becomes the $k$-algebra $E_{0,0}^\lambda$ in Definition \ref{def:Edeformations}.
%	generated by
%	$a,b,c,g^{\pm1},h^{\pm1}$ so that $gh=hg$ and with $\lambda$-commutation relations:
%	\begin{align}\label{eqn:relations-basic-def0}
%		ha=-\lambda^{-2}ah, \quad   h b=-bh, \quad h c=-ch, \qquad   ga=-ag, \quad  g b=-\lambda^2cg, \quad  g c=-\lambda^2bg
%	\end{align}
%	together with the $\lambda$-quadratic identities
%	\begin{align}\label{eqn:relations-quad-def0}
%		ab-\lambda^4ba&=0, & ac-\lambda^4ca&=0, & bc-cb&=0.
%	\end{align}
	In turn, $H\#_{\theta_\tau} U\simeq { _{\hat{\tau}}\!(H^{}\tb U)}$ is the algebra $A_{0,0}^\lambda$ as stated. 
%	by $a,b,c,g^{\pm1},h^{\pm1}$ with relations \eqref{eqn:relations-basic-def0} and \eqref{eqn:relations-quad-def0}, with the caveat 
%	\[
%	gh=\lambda\,hg.
%	\]
\epf
\end{example}

Next we produce a more involved pairing (which includes the previous one).

\begin{example}\label{exa:tau1}
	Fix $\beta\in k$ and $\xi\in \{\pm1\}$.
	Consider the linear map $\tau_\xi^{\beta}\colon H\ot U\to k$ given by:
	\begin{align}\label{eqn:tau:exa}
		\tau_\xi^\beta(b^rc^sh^p,a^ng^q)&=\delta_{r+s,n}(-1)^{qn}\xi^{pq}n!\beta^n,   
	\end{align}
	for every $r,s,n\in\N$, $p,q\in\Z$. We write $\tau_{\pm}^\beta\coloneqq \tau_{\pm1}^\beta$. Then $\tau=\tau^\beta_{\xi}$ is a skew $H\tb U$-pairing.
	
	It follows that $H\overset{\tau} \tb U\simeq E_{0,\beta}^\pm$ and $H\#_{\theta_\tau} U\simeq A_{0,\beta}^\pm$.
	%\begin{claim}
	%$\tau=\tau^\beta_{\xi}$ is a skew $H\tb U$-pairing.
	%\end{claim}

\begin{remark}
If $\beta=0$, then this is Example \ref{exa:tau0} with $\lambda=\xi$. In particular, we can unify Examples \ref{exa:tau0} and \ref{exa:tau1} in a single pairing $\tau_\lambda^\beta$, where $\lambda\in k^\times$ and $\beta\in k$ are subject to the second restriction in \eqref{eqn:conditions-abl}; namely $\beta=0$ if $\lambda\notin\{\pm1\}$.
\end{remark}
			
	\begin{proof}
		Observe first that $\tau_{|H_i\ot U_j}=0$ for $i\neq j$ and $\tau(g,h)=\xi=\pm 1$. In particular, it is convolution invertible.
		We start with the second identity \eqref{eqn:tauskew-restricted}: We proceed by induction on $\sigma\coloneqq s+t+n\geq 0$, for $f\in H_s$, $f\in H_t$, $u\in U_n$.
		
		The case $\sigma=0$ or $s=t=n=0$ follows since this defines algebra maps $\tau^h$ and $\tau^g$ as in Remark \ref{rem:quick}. Now assume $\sigma>0$. Then 
		$\tau(ff',u)=0$. As for the right hand side of the equation, we have two cases: $s+t=n$ and $s+t\neq n$. We start with the later. If $s=0$ (so $t\neq n$), we may assume $f\in\lg h\rg$ and then
		\[
		\tau(f,u_{(1)})\tau(f',u_{(2)}\tl r(f))=\tau(f,u_{(1)})\tau(f',u_{(2)}\tl f).
		\]
		If $n=0$ then $t>0$ and $\tau(f',-)=0$. 
		
		If $n>0$, then $\tau(f,u_{(1)})\tau(f',u_{(2)}\tl f)=\tau(f,\ell(u))\tau(f',u\tl f)=0$ since $\tau(-,u\tl f)=0$.
		Now, if $s>0$, then $\tau(f,-)=0$ and thus $\tau(f,u_{(1)})\tau(f',u_{(2)}\tl r(f))=0$.
		
		Let us now assume that $s+t=n$, $f=b^{s_1}c^{s_2}h^{p_1}$, $f'=b^{t_1}c^{t_2}h^{p_2}$, $u=a^ng^q$, with $s_1+s_2=s$, $t_1+t_2=t$ and $p_1,p_2,q\in\Z$. We have that
		\begin{align}\label{eqn:lefthand}
			\tau(ff',u)=(-1)^{p_1t}(-1)^{qn}\xi^{(p_1+p_2)q}n!\beta^n.
		\end{align}
		On the other hand,
		\begin{align*}
			\tau(f,u_{(1)})\tau(f',u_{(2)}\tl r(f))=\tau(f,(a^n)_{(1)}g^q)\tau(f',(a^n)_{(2)}\tl h^{p_1}g^q).
		\end{align*}
		The component of $\Delta(a^n)=(a^n)_{(1)}\ot (a^n)_{(2)}=\sum_{i= 0}^n\binom{n}{i}a^ig^{2(n-i)}\ot a^{n-i}$ in the right component $U_s\ot U_t$ is precisely $\binom{n}{s}a^{s}g^{2t}\ot a^t$. So the above equation is 
		\begin{align*}
			\binom{n}{s}\tau(f,a^{s}g^{2t}g^q)\tau(f',a^t\tl h^{p_1}g^q)&=\frac{n!}{s!t!}(-1)^{qs}\xi^{(2t+q)p_1} s!\beta^{s} (-1)^{p_1t}(-1)^{qt}\xi^{qp_2}t!\beta^t\\
			&=n!(-1)^{q(s+t)}\xi^{(p_1+p_2)q}(-1)^{p_1t}\beta^{s+t}
		\end{align*}
		which coincides with \eqref{eqn:lefthand}. As for the first identity, if $u=a^ng^q$, $u'=a^mg^r$ and $f=b^sc^th^p$, then the left hand side is
		\begin{align*}
			\tau(f,uu')&=(-1)^{qm}\tau(b^sc^th^p,a^{n+m}g^{q+r})=(-1)^{qm}(-1)^{(q+r)(n+m)}\delta_{s+t,n+m}\xi^{p(q+r)}(n+m)!\beta^{n+m}\\
			&=(-1)^{qn+rn+rm}\delta_{s+t,n+m}\xi^{p(q+r)}(n+m)!\beta^{n+m},
		\end{align*}
		while as $\Delta(f)=\Delta(b^s)\Delta(c^t)(h^p\ot h^p)=\sum_{i=0}^s\sum_{j=0}^t\binom{s}{i}\binom{t}{j} b^ic^jh^{2(s+t-i-j)+p}\ot b^{s-i}c^{t-j}h^p$, then we get that
		$\tau(f_{(1)},u')\tau(r(u')\tr f_{(2)},u)=\tau(f_{(1)},u')\tau(g^{r}\tr f_{(2)},u)$ equals:
		\begin{align*}
			\sum_{i=0}^s\sum_{j=0}^t&\binom{s}{i}\binom{t}{j}\delta_{i+j,m}(-1)^{rm}\xi^{pr}m!\beta^m (-1)^{r(i+j)}\delta_{s+t-i-j,n}(-1)^{qn}\xi^{pq}n!\beta^{n}\\
			&=\sum_{i=0}^s\sum_{j=0}^t\binom{s}{i}\binom{t}{j}\delta_{i+j,m}m!n!\delta_{s+t,n+m}(-1)^{qn+rn+rm}\xi^{p(q+r)}\beta^{n+m}.
		\end{align*}
		This coincides with the expression for the left hand side computed above as the coefficient $\sum_{i=0}^s\sum_{j=0}^t\binom{s}{i}\binom{t}{j}\delta_{i+j,m}$ counts the number of subsets of size $i+j=m$ in a set of size $s+t$ (by choosing $i$ from a subset of size $s$ and $j$ from the complement of size $t$), namely this coefficient is $\binom{s+t}{m}$. Now $\binom{s+t}{m}m!n!\delta_{s+t,n+m}=(m+n)!$ and both sides coincide. Thus the first part follows.

	Now, the corresponding cocycle $\hat{\tau}=\hat{\tau}_{\pm}^\beta\colon H\tb U\ot H\tb U\to k$ can be computed explicitly using \eqref{eqn:tau:exa}, via the definition in Lemma \ref{lem:tauskewcocycle}. In particular, it is concentrated on $U\ot H$, and, for our purposes, it is enough to determine its values on the generators $\{a,g\}\times \{b,c,h\}$. We easily get:
	\begin{align*}
		\hat{\tau}(a,b)&=\hat{\tau}(a,c)=\beta, & \hat{\tau}(g,h)&=\pm1, & \hat{\tau}(a,h)&=\hat{\tau}(g,b)=\hat{\tau}(g,b)=0.
	\end{align*}
	Therefore, we obtain that $H\overset{\tau} \tb U\simeq (H\tb U)^{\hat{\tau}}$ is the $k$-algebra $E_{0,\beta}^\pm$ from Definition \ref{def:Edeformations}.
%	generated by $a,b,c,g^{\pm1}, h^{\pm1}$ with the diagonal relations in \eqref{eqn:relations-basic} and the {\it deformed} relations:
%	\begin{align}\label{eqn:relations-quad-def1-tau}
%		ab-ba&=\beta(1-g^2h^2), & ac-ca&=\beta(1-g^2h^2), & bc-cb&=0.
%	\end{align}
	Indeed, this is a standard computation for the multiplication in $(H\tb U)^{\hat{\tau}}$, see \eqref{eqn:cocycle-def}:
	\begin{align*}
		a.b&=\sigma(a_{(1)},b_{(1)})a_{(2)}b_{(2)}\sigma^{-1}(a_{(3)},b_{(3)})=\sigma(a,b)1+\sigma(g^2,h^2)ab+\sigma(g^2,h^2)g^2h^2\sigma^{-1}(a,b)\\
		&=\beta+ab-\beta gh=\beta(1-g^2h^2).
	\end{align*}
	As $b.a=ba$, we obtain the first (deformed) quadratic relation for $E_{0,\beta}^\pm$. The others follow similarly.
	
	As well, we can analogously  describe the Hopf-Galois object $H\#_{\theta_\tau} U\simeq \,_{\hat{\tau}}\!H\tb U$, with multiplication \eqref{eqn:cocycle-def-galois}: we obtain the  
	$k$-algebra $A_{0,\beta}^\pm$.
%	generated 
%	by $a,b,c,g^{\pm1}, h^{\pm1}$ with diagonal relations as in \eqref{eqn:relations-basic} together with
%	\begin{align}\label{eqn:relations-quad-def2-tau}
%		gh&=\pm hg & ab-ba&=\beta, & ac-ca&=\beta, & bc-cb&=0.
%	\end{align}
	Notice that in this setting $b.a=ab$ and $a.b=\sigma(a_{(1)},b_{(1)})a_{(2)}b_{(2)}=\sigma(a,b)1+\sigma(g^2,h^2)ab=\beta+ab$.
		This concludes the example.
		\epf
\end{example}

\medbreak 

Next we show that these are indeed all the skew-parings for $H\tb U$. 

\begin{proposition}\label{prop:all-skew}
	Let $\tau\colon H\ot U\to k$ be a $H\tb U$-skew-paring. Then there is $\lambda\neq 0$ so that $\tau=\tau_\lambda$ or there is $\beta\in k$ so that  $\tau=\tau_\pm^\beta$. 
\end{proposition}
%Recall that we allow the overlapping $\tau_{\pm 1}=\tau_\pm^0$.

\pf
Let $\lambda=\tau(g,h)$. Then $\lambda\neq 0$ since $\tau$ is convolution invertible. 
Following \eqref{eqn:tauskew-restricted}, we get:
\begin{align*}
	\tau(bh,a)&=\tau(b,a), & \tau(hb,a)=-\lambda^2\tau(b,a).
\end{align*}
and we see that $\lambda^2=1$ or $\tau(b,a)=0$, since $bh=-hb$.
On the other hand, 
\begin{align*}
	\tau(b,ag)&=-\lambda^2\tau(c,a)=-\tau(c,a), & \tau(b,ga)=\tau(b,a).
\end{align*}
As $ag=-ga$ we get that $\tau(c,a)=\tau(b,a)$. In particular have that either
\begin{align*}
	\lambda^2&=1 \text{ and } \tau(c,a)=\tau(b,a) &\text{ or }&&	\tau(b,a)&=\tau(c,a)=0.
\end{align*}

Similarly, if $n>0$ and $p,q\in\Z$, then we easily deduce from \eqref{eqn:tauskew-restricted} that $\tau(h^p,a^ng^q)=\lambda^{pq}\tau(h^p,a^n)$ and $\tau(h^p,g^qa^n)=\lambda^{pq}\tau(h^p,a^n)$. As $ga=-ag$, this implies that $\tau(h^p,a^n)=0$. This idea can be repeated to show that, more generally:
\begin{align}\label{eqn:1-0}
	\tau_{|H_0\ot U_n+H_n\ot U_0}=0, \qquad n>0.
\end{align}
But now  \eqref{eqn:tauskew-restricted} implies $\tau_{|H_i\ot U_j}=0$ when $i\neq j$.

Hence we deduce that when $\tau(b,a)=\tau(c,a)=0$ we have that $\tau=\tau_{\lambda}$ as in Example \ref{exa:tau0}. 

Otherwise, we have that $\tau(g,h)=\lambda$ is such that $\lambda^2=1$. We set $\xi=\lambda\in\{1,-1\}$, $\beta=\tau(b,a)=\tau(c,a)$. 
We can use \eqref{eqn:tauskew-restricted} to check that, for $r+s=n\in\N$:
\begin{align}\label{eqn:factorial}
	\tau(b^n,a^n)=\tau(c^n,a^n)=\tau(b^rc^s,a^n)=n!\beta^n.
\end{align}
Fix now $f=f(b,c)\in k[b,c]$ of degree $n\in\N$ (notice $r(f)=1$, $\ell(f)=h^{2n}$) and $p,q\in \Z$. Then
\begin{align*}
	\tau(fh^p,a^ng^q)&=\tau(f,(a^n)_{(1)}g^q)\tau(h^p,(a^n)_{(2)}g^q)=\tau(f,a^ng^q)\tau(h^p,g^q)=\xi^{pq}\tau(f,a^ng^q)\\
	&=\xi^{pq}\tau(f_{(1)},g^q)\tau(g^q\tr f_{(2)},a^n)=\xi^{pq}\tau(h^{2n},g^q)\tau(g^q\tr f,a^n)=\xi^{pq}\xi^{2qn}\tau(f_q,a^n).
\end{align*}
Here $f_q=f(b_q,c_q)\in k[b,c]$, for $b_q=\begin{cases}
	b, & q \text{ even,} \\	c, & q \text{ odd,}
\end{cases}$ and $c_q=\begin{cases}
	c, & q \text{ even,} \\	b, & q \text{ odd.}\end{cases}$

%\begin{align}\label{eqn:bq}
%	b_q&=\begin{cases}
%		b, & q \text{ even,} \\	c, & q \text{ odd,}
%	\end{cases} & c_q&=\begin{cases}
%		c, & q \text{ even,} \\	b, & q \text{ odd.}
%\end{cases}\end{align} 
Now, by \eqref{eqn:factorial}, we have $
\tau(fh^p,a^ng^q)=\xi^{pq}\xi^{2qn}\tau(f_q,a^n)=(-1)^{qn}\xi^{pq}n!\beta^n$.
Therefore $\tau=\tau_{\pm}^\beta$ is as in \eqref{eqn:tau:exa} in Example \ref{exa:tau1}, according to $\xi=\pm 1$. The lemma follows.
\epf

\subsection{Hopf-Galois objects}\label{subsec:examples-psi-twisting-cleft}
We now describe the Hopf-Galois objects over $E = H  \tb U$.
%We now illustrate the previous results in the case of our example from Subsection \ref{sec:exbicrossed}.
%Recall the Hopf algebras $H$ and $U$ there. 
%These Hopf algebras are pointed, hence  any Hopf-Galois object is cleft, and  the antipodes $S_{H}$ and $S_U$ are invertible.

The cleft objects $A$ for $H$ that produce a non-trivial left Hopf algebra, namely $L(A,H)\not\simeq H$, are the (pairwise isomorphic) algebras $H_\alpha$,  $\alpha\neq 0$; here $H_\alpha$ is the $k$-algebra generated by $\{b,c,h^{\pm1}\}$ with relations
\begin{align}\label{eqn:H-alpha}
	hb&=-bh, & hc&=-ch, & bc-cb=\alpha.
\end{align}
The Hopf algebra $L_\alpha=L(H,H_\alpha)$ is the {\it deformation}
generated by $\{b,c,h^{\pm1}\}$ with relations
\begin{align}\label{eqn:L-alpha}
	hb&=-bh, & hc&=-ch, & bc-cb=\alpha(1-h^2).
\end{align}
In turn, there are no Hopf deformations of $U$, hence there are no cleft objects for $U$ in this sense. We shall look, however, into non-trivial cocycles for $U$ in \S\ref{sec:unrolled} at the end of this monograph, in a different setting. As well, we analyze the case $\alpha=0$ in Remark \ref{rem:more-twist} below.

For any $\alpha\in k$, the cocycle $\sigma_\alpha\colon H\ot H\to k$ which determines the deformation $H_\alpha$ is 
\begin{align}\label{eqn:sigma-full}
	\sigma_\alpha(b^nc^mh^p,b^rc^sh^q)=\delta_{n,s}\delta_{m,r}(-1)^m(-1)^{p(n+m)}n!m!\left(\frac{\alpha}{2}\right)^{n+m},
\end{align}
for $n,m,r,s\geq 0$, $p,q\in\Z$. We shall provide some insight for this formula on Remark \ref{rem:recipe}, whereas we refer to \cite{gam} for more details. 

We will be primarily interested in the following consequences for $\sigma=\sigma_\alpha$:
\begin{align}
	\begin{split}
		\label{eqn:sigma}
		\sigma_{|H_i\ot H_j}&=0, \ i\neq j, \qquad  \sigma_{|H_0\ot H_0}=\eps\ot\eps, \\
		\sigma(bh^p,ch^q)&=-\sigma(ch^p,bh^q)=(-1)^p\frac{\alpha}{2}, \qquad  \sigma(bh^p,bh^q)=\sigma(ch^p,ch^q)=0,	\end{split}
\end{align}
As well, we have that $\sigma$ is skew-linear with respect to  $\lg g\rg$, as:
\begin{align}\label{eqn:sigmaq} 
	\sigma(g^q\tr f, g^q\tr f')&=(-1)^{qn}\sigma(-,-), \qquad f,f'\in H_n.
\end{align}

\begin{remark}\label{rem:recipe}
	The recipe for the full computation of $\sigma$ is as follows: let us set $V=k\{b,c\}$, $\Z=\lg h\rg$:  consider the $\Z$-invariant map $\eta'\colon V\ot V\to k$:
	\begin{align*}
		\eta'(b,c)&=\alpha/2, &  \eta'(c,b)&=-\alpha/2, &  \eta'(b,b)=\eta'(c,c)=0
	\end{align*}
	and extend it to an $\eps$-derivation $\eta\colon H\ot H\to k$ via:
	\begin{align*}
		\eta(v_1h^p,v_2h^q)&=\eta(v_1,h^p\cdot v_2), \ v_1,v_2\in V,\, p,q\in\Z, &  \eta_{|H_i\ot H_j}=0, \  (i,j)\neq (1,1).
	\end{align*}
	Then $\sigma$  arises as the exponential $e^\eta=\sum\limits_{n\geq 0}\frac{1}{n!}\eta^{\ast n}$; here $\eta^{\ast n}$ stands for the $n$th power of the convolution product $\ast$, namely $\underbrace{\eta\ast \dots\ast \eta}_{n\text{ times}}$. The formula follows from the projection onto $V^{\ot n+m}$ of the iteration $\Delta^{(n+m-1)}(b^nc^m)$.
\end{remark}

Next we investigate the maps $\psi\colon H\ot U\to k$ that satisfy \eqref{eqn:twistconditions-psi}, in order to compute the corresponding twisting maps $\theta=\theta_\psi$. 

Notice right away that as one of the cocycles is trivial and both antipodes are invertible, then any such $\psi$ is convolution invertible, by Remark \ref{rem:psi-invertible}. Moreover, the bijectivity of the antipodes implies that $\theta_\psi$ is invertible as well.
Recall the notation $r(-)$, $\ell(-)$ in \eqref{eqn:Delta-rf}. Now, arguing as in \S\ref{sub:exskew} we use once again \eqref{eqn:action-zero} to simplify  \eqref{eqn:twistconditions-psi} into: \begin{align}\label{eqn:twistconditions-psi-restricted}
	\psi(f,uu')&=\psi(f_{(1)},u')\psi(r(u') \tr f_{(2)},u),\\
	\notag	\sigma(f_{(1)},f'_{(1)})\psi(f_{(2)}f'_{(2)},u)&=\psi(f_{(1)},u_{(1)})\psi(f'_{(1)},u_{(2)}\tl r(f_{(1)}))\sigma(r(u_{(2)})\tr f_{(2)},r(u_{(2)}) \tr f'_{(2)}),
\end{align}
for every $f\in H_s$, $f\in H_t$, $u\in U_n$, $u'\in U_m$, $s,t,n,m\geq 0$.

\medbreak

We obtain similar results as in \S\ref{sub:exskew}, though more restrictive; see Remark \ref{rem:more-twist} below.

%We start with a remark.
%\begin{rem}
%Fix $\alpha\neq 0$. The algebra $H_\alpha$ is not graded, but it does admit an increasing filtration $H_{\alpha}^{(n)}=k[b^rc^s:r+s\leq n]\# k\Z$,  $n\geq 0$ so that $H_\alpha=\bigcup_{n\geq 0} H_\alpha^{ (n)}$. This is compatible with the coaction $\rho\colon H_\alpha\to H_\alpha\ot H$, meaning that $\rho(H_{\alpha}^{ (n)})\subseteq \sum_{i=0}^n H_{\alpha}^{ (i)}\ot H_{n-i}$, for each $n\geq 0$. It coincides with the so-called {\it Loewy filtration} of $(H_\alpha,\rho)$.
%\end{rem}

\begin{lemma}\label{lem:all-psi}
	Fix $\zeta \in k$, with $\zeta^4=-1$. 
	\begin{enumerate}
		\item Let $\psi=\psi_\zeta\colon H\ot U\to k$ be the linear map
		\begin{align*}
			\psi(h^p,g^q)&=\zeta^{pq}, & \psi_{|H_i\ot U_j}=0, \ i+j>0.
		\end{align*}
		Then $\psi$ satisfies \eqref{eqn:twistconditions-psi} and $\theta_\zeta\coloneqq\theta_\psi\colon U\ot H_\alpha \to  H_\alpha\ot U$ is such that
		\begin{align*}
			\theta_\zeta(g\ot h)&=\zeta\,h\ot g, & \theta_\zeta(g\ot b)&=-\zeta^2c\ot g, & \theta_\zeta(g\ot c)&=-\zeta^2b\ot g, \\
			\theta_\zeta(a\ot h)&=-\zeta^2 h\ot a, & \theta_\zeta(a\ot b)&=\zeta^4 b\ot a, & \theta_\zeta(a\ot c)&=\zeta^4 c\ot a.
		\end{align*}
%		The algebra $H_\alpha\#_{\theta_\zeta} U$ is 
%		generated by $\{a,b,c,h^{\pm1},g^{\pm1}\}$ so that $gh=\zeta hg$ and determined by the commutation relations:
%		\begin{align}\label{eqn:relations-basic-def1}
%			ha=-\zeta^{-2}ah, \quad   h b=-bh, \quad  h c=-ch,  \qquad  ga=-ag, \quad  g b=-\zeta^2cg, \quad  g c=-\zeta^2bg,
%		\end{align}
%		together with the following deformation of \eqref{eqn:relations-quad-def0}:
%		\begin{align}\label{eqn:relations-quad-def1-psi}
%			ab-\zeta^4ba&=0, & ac-\zeta^4ca&=0, & bc-cb&=\alpha.
%		\end{align}
%		%\item Let $\psi=\psi_{\xi}^\beta\colon H\ot U\to k$ be the linear map
		%\begin{align}\label{eqn:psi:exa}
		%	\psi_{\xi}^\beta(b^rc^sh^p,a^ng^q)&=\xi^{pq}\delta_{r+s,n}(-1)^{qn}n!\beta^n.
		%\end{align}
		%Then $\psi$ satisfies \eqref{eqn:twistconditions-psi}. We set $\psi^\beta_\pm=\psi^\beta_{\chi}$ and $\theta^\beta_\pm\coloneqq\theta_\psi\colon U\ot H_\alpha \to  H_\alpha\ot U$ is such that
		%\begin{align*}
		%	\theta^\beta_\pm(g\ot h)&=\pm h\ot g, & \theta^\beta_\pm(g\ot b)&=-c\ot g, & \theta^\beta_\pm(g\ot c)&=-b\ot g, \\
		%	\theta^\beta_\pm(a\ot h)&=-h\ot a, & \theta^\beta_\pm(a\ot b)&=\beta\,1\ot 1+b\ot a, & \theta^\beta_\pm(a\ot c)&=\beta\,1\ot 1+c\ot a.
		%\end{align*} 
		%The algebra $H_\alpha\#_{\theta^\beta_\pm} U$ is determined by \eqref{eqn:relations-basic} and the deformation
		%\begin{align}\label{eqn:relations-quad-def2-psi}
		%gh&=\pm hg, &	ab-ba&=\beta, & ac-ca&=\beta, & bc-cb&=\alpha.
		%\end{align}
		\item Conversely, if $\psi\colon H\ot U$ is a  linear map satisfying \eqref{eqn:twistconditions-psi}, then $\psi=\psi_\zeta$.
	\end{enumerate}
It follows that $H_\alpha\#_{\theta_\zeta} U$ is $A_{\alpha,0}^\zeta$ and $L_\alpha \btb U\simeq E_{\alpha,0}^\zeta$.
\end{lemma}

\pf
The lemma is essentially analogous to Example \ref{exa:tau0}, combined with a partial Proposition \ref{prop:all-skew}. The current setting only forces the restriction $\zeta^4=-1$ on the map concentrated on degree zero. We sketch a proof, following the ideas in the cases just mentioned.

(1) The first equation in \eqref{eqn:twistconditions-psi-restricted} follows automatically as the corresponding \eqref{eqn:tauskew-restricted} in Example \ref{exa:tau0}.
As for the second, let $e,e'\in k[b,c]$ be homogeneous elements of degree $d\geq 0$, $r,s\in\Z$ and $u=g^q\in U_0$ (if $u\in U_{>0}$, then both sides are zero). Set $f=eh^r$, $f'=e'h^s$. Then this equation is equivalent to
\begin{align}\label{eqn:zeta4}
	\sigma(f,f')\psi(h^{r+s},u)=(-1)^{qd}\psi(h^{2d+r},g^q)\psi(h^{2d+s},g^q)\sigma(f,f'),
\end{align}
namely $1=(-1)^{qd}\zeta^{4qd}, \forall\,d\in\N$; which holds when (and only if) $\zeta^4=-1$.

As for (2), this is a consequence of the skew-linearity of $\sigma$ as in \eqref{eqn:sigmaq}. We start by imitating the first part of  Proposition \ref{prop:all-skew}. 
Indeed, let $\psi\colon H_\alpha\ot U\to k$ satisfying \eqref{eqn:twistconditions-psi-restricted}. As $\sigma_{|H_0\ot H+H\ot H_0}=\eps\ot\eps$, we can follow the first lines of loc.cit.~to conclude that $\zeta\coloneqq \psi(h,g)\neq 0$. As well, we get once again that either 
$\zeta^2=1$ and $\psi(b,a)=\psi(c,a)$ or else $\psi(b,a)=\psi(c,a)=0$. For this last case we deduce that $\tau=\tau_{\zeta}$ as in (1), as we have observed above that the identity $\zeta^4=-1$ is  necessary. 

Otherwise, $\tau(g,h)=\zeta$, with $\zeta^2=1$. We can look at equation \eqref{eqn:zeta4} in this context for $u=g$; which still represents the second condition in \eqref{eqn:twistconditions-psi-restricted}.
If we set $f=b$, $f'=c$ we obtain the contradiction $\sigma(b,c)=-\sigma(b,c)$. This shows (2).

In turn, recall that $Q=L(U,U)\simeq U$ and $L_\alpha=L(H_\alpha,H)$ is as in \eqref{eqn:L-alpha}.
In this setting we get that $L_\alpha \btb U$ becomes the $k$-algebra $E_{\alpha,0}^\zeta$ in Definition \ref{def:Edeformations}.
\epf

We obtain more examples if we let $\alpha=0$, as we return to the case in the previous section. We include them to make the twisting map explicit following the lines of this section.
\begin{remark}\label{rem:more-twist}
	If we choose $\alpha=0$, namely $H_\alpha=H$ is the trivial cleft object, then $\sigma=\eps$ is trivial as well and equations \eqref{eqn:twistconditions-psi} become \eqref{eqn:tauskew}; namely $\psi\colon H\ot U\to k$ is a $H\tb U$-skew-paring. Hence we recover Examples \ref{exa:tau0} and \ref{exa:tau1}, by setting $\psi_\lambda\coloneqq \tau_{\lambda}$ and $\psi_{\pm}^\beta\coloneqq \tau_{\pm}^\beta$. Namely:
	
	\begin{enumerate}[leftmargin=*]
		\item The first is an unrestricted $\psi=\psi_\lambda$, and the corresponding $\theta_\lambda$, from Lemma \ref{lem:all-psi}, as the condition $\lambda^4=-1$ comes from the fact that $\alpha\neq0$ therein. 
		As for the Galois objects, if we set $\theta=\theta_\lambda$, then $H \btb U\simeq H\tb U$ is the $k$-algebra $A_{0,0}^\lambda$.
%		generated by $\{a,b,c,g^{\pm1},h^{\pm1}\}$ with basic $\lambda$-commutation \eqref{eqn:relations-basic-def0}
%		together with the $\lambda$-quadratic identities
%		\eqref{eqn:relations-quad-def0}.
		
		\item As for the second, if $\psi=\psi_{\pm}^\beta$, then the twisting map $\theta^\beta_\pm\coloneqq\theta_\psi\colon U\ot H \to  H\ot U$ is:
		\begin{align*}
			\theta^\beta_\pm(g\ot h)&=\pm h\ot g, & \theta^\beta_\pm(g\ot b)&=-c\ot g, & \theta^\beta_\pm(g\ot c)&=-b\ot g, \\
			\theta^\beta_\pm(a\ot h)&=-h\ot a, & \theta^\beta_\pm(a\ot b)&=\beta\,1\ot 1+b\ot a, & \theta^\beta_\pm(a\ot c)&=\beta\,1\ot 1+c\ot a.
		\end{align*} 
	\end{enumerate}
	The corresponding algebras $H\#_\theta U$ naturally coincide with those in Examples \ref{exa:tau0} and \ref{exa:tau1}. 
	In turn, if $\theta=\theta_\pm^\beta$, then $H \btb U\simeq E_{0,\beta}^{\pm}$.
%	 is  generated by $\{a,b,c,g^{\pm1},h^{\pm1}\}$ with basic commutation \eqref{eqn:relations-basic}
%	and the deformation \eqref{eqn:relations-quad-def1-tau}, namely:
%	\begin{align*}
%		ab-ba&=\beta(1-gh), & ac-ca&=\beta(1-gh), & bc-cb&=0.
%		\end{align*}	
	\end{remark}

\subsection{Conclusions} We summarize the results in this section in the following statement. Recall the definition of the algebras $E_{\alpha,\beta}^{\lambda}$ and $A_{\alpha,\beta}^\lambda$ be as in Definition \ref{def:Edeformations}. As well, recall the examples $\tau_{\lambda}$ and $\tau^{\beta}_\pm$ of skew-pairings $H\ot U\to k$ from Examples \ref{exa:tau0} and \ref{exa:tau1}, respectively.

We recall as well the Hopf cocycles $\sigma_\alpha\colon H\ot H\to k$ as in \eqref{eqn:sigma-full} and the corresponding deformations $H_\alpha$ in \eqref{eqn:H-alpha} and $L_\alpha$ in \eqref{eqn:L-alpha}; in particular $H_0=L_0=H$. Finally, we shall consider the maps $\psi_\zeta\colon H\ot U\to k$ as in Lemma \ref{lem:all-psi}, with $\zeta^4=-1$. 

\begin{corollary}\label{cor:summarize}
	Let $E=H\tb U$ be as in Definition \ref{def:E}.
\begin{enumerate}[leftmargin=*]
\item Let $\tau\colon H\ot U\to k$ be a skew-pairing; then $\tau=\tau_{\lambda}$ or $\tau=\tau^{\beta}_\pm$. 
	\begin{enumerate}[leftmargin=*]
	\item If $\tau=\tau_{\lambda}$, then $H\overset{\tau} \tb U\simeq E_{0,0}^\lambda$ and $H\#_{\theta_\tau} U\simeq A_{0,0}^\lambda$.
	\item If $\tau=\tau_{\beta}^\pm$, then $H\overset{\tau} \tb U\simeq E_{0,\beta}^\pm$ and $H\#_{\theta_\tau} U\simeq A_{0,\beta}^\pm$.
	\end{enumerate}	
\item Let $\psi\colon H\ot U\to k$ be a map satisfying \eqref{eqn:twistconditions-psi} for $\sigma=\sigma_\alpha$ and $\tau=\eps_U^{\ot2}$.
	\begin{enumerate}[leftmargin=*]
	\item If $\alpha\neq 0$, then $\psi=\psi_\zeta$. We get $L_\alpha \btb U\simeq E_{\alpha,0}^\zeta$ and $H_\alpha\#_{\theta_\zeta} U\simeq A_{\alpha,0}^\zeta$.
	\item If $\alpha=0$ then $\psi$ is a skew-pairing $\tau^{\beta}_{\lambda}$. 
	Hence $L_\alpha \btb U\simeq E_{0,\beta}^\lambda$ and $H_\alpha\#_{\theta} U\simeq A_{0,\beta}^\lambda$.
\end{enumerate}	
\end{enumerate}
\end{corollary}

In the language of Definition \ref{def:Edeformations}, we could further condense the information in the examples from this section by saying that algebras in the case $\beta\neq 0$ come from skew-pairings and the case $\alpha\neq 0$ comes from linear maps and cocycles in the cleft setting.

\pf
The first assertion in (1) is Proposition \ref{prop:all-skew}. Then (a) and (b) therein follow from the corresponding Examples \ref{exa:tau0} and \ref{exa:tau1}. In turn, (2)(a) is Lemma \ref{lem:all-psi} and (b) is Remark \ref{rem:more-twist}.
\epf

\section{Hopf-Galois objects over a semi-direct product Hopf algebra}\label{sec:semi}

We specialize the results of the previous section to the semi-direct product Hopf algebra case, i.e. when one of the actions in the matched pair is trivial. We will see that in that case, we can make Theorem \ref{thm:bicrossedtwist1} and Theorem \ref{thm:bicrossedtwist2} much more precise, at least when we assume that all the Galois objects over the Hopf algebra $U$ are cleft.

\subsection{General results}
The framework  in this section is that of a matched pair of Hopf algebras $(H,U,\triangleright, \triangleleft)$ with $\triangleleft$ trivial. The matched pair conditions become that $H$ is a left $U$-module algebra together with
\begin{align} \label{eqn:cocosemi} 
	x_{(1)} \otimes x_{(2)}\triangleright a = x_{(2)} \otimes x_{(1)}\triangleright a, \qquad x\in U, a \in H.
\end{align}
The resulting Hopf algebra is denoted $H\rtimes U$. As an algebra it is just a smash product as before, but to stress out the Hopf structure, we will call it a semi-direct product Hopf algebra. 

 In view of the above condition (\ref{eqn:cocosemi}), the most interesting semi-direct products (i.e.~those with $\triangleright$ non trivial) will occur with $U$ cocommutative, and hence in particular with $U$ having the property that any Galois object is cleft, an assumption that will be made in the main results of the section.

We begin with a  basic construction. This is certainly well-known, and the  straightforward proof is left to the reader.

\begin{lemma}\label{lem:twistedsmash}
	Let $U$ be a Hopf algebra, let $\sigma : U\otimes U \to k$ be a $2$-cocycle and let $A$ be a $U$-module algebra. Then the map
	\begin{align*}
	\theta : {_\sigma  U } \otimes A &\to A\otimes {_\sigma U} &
	x \otimes a &\mapsto x_{(1)}\cdot a\otimes x_{(2)}&
	\end{align*}
	is a twisting map. The resulting algebra $A\#{_\sigma U}$ is called a twisted smash product.\qed
\end{lemma}

We need one more preparatory result. Again the proof is an immediate verification, that we leave to the reader.

\begin{lemma}\label{lem:compatsemi}
	Let $E = H\rtimes U$ be a semi-direct product Hopf algebra, let $\sigma : U\otimes U \to k$ be a $2$-cocycle and let $A$ be a $U$-module algebra and a right $H$-comodule algebra. Then the twisting map 	$\theta : {_\sigma  U } \otimes A \to A\otimes {_\sigma U}$ in Lemma \ref{lem:twistedsmash} makes Diagram (\ref{eqn:diagram-galois}) commute (with $R={_\sigma U})$ if and only if the right coaction $\rho_A : A \to A\otimes H$ is left $U$-linear. \qed 
\end{lemma}

\begin{theorem}\label{thm:galoissemidirect}
	Let $E = H\rtimes U$ be a semi-direct product Hopf algebra.
	\begin{enumerate}
		\item Let $\sigma : U\otimes U \to k$ be a $2$-cocycle and let $A$ be a right $H$-Galois object. Assume that $U$ acts on $A$ as a $U$-module algebra and that the right coaction $\rho_A : A \to A\otimes H$ is left $U$-linear. Then the twisted smash product $A\#{_\sigma U}$ is a right $E$-Galois object.
		\item Conversely, assuming that any right $U$-Galois is cleft, any right $E$-Galois object is isomorphic to a twisted smash product $A\#{_\sigma U}$ as above. 
	\end{enumerate}
\end{theorem}

\begin{proof}
	The first assertion is a direct consequence of Lemma \ref{lem:compatsemi} and of (the first part of) Theorem \ref{thm:bicrossedtwist1}. Let $J$ be a right $E$-Galois object. By the second part of Theorem \ref{thm:bicrossedtwist1}, we have $J \simeq A\#_\theta R$ where $A$ is a right $H$-Galois object, $R$ is a right $U$-Galois object, and $\theta : R\otimes A\to A\otimes R$ is a twisting map making Diagram (\ref{eqn:diagram-galois}) commute. By our assumption on $U$, we can assume that $R={_\sigma U}$ for some $2$-cocycle $\sigma : U\otimes U \to k$. The commutativity of Diagram (\ref{eqn:diagram-galois}) implies that
	$$\theta : U\otimes A \to A\otimes U$$ 
	is right $U$-colinear (where $U\otimes A$ and $A\otimes U$ have the comodule structure induced by the comultiplication of $U$). A standard argument then shows that there is a linear map $\mu : U\otimes A \to A$ such that 
	$$\theta = (\mu\otimes{\rm id}_U) \circ \alpha_{U\otimes A}$$
		with $\alpha_{U\otimes A}(x\otimes a)=x_{(1)}\otimes a \otimes x_{(2)}$. Thus if we denote $\mu(x\otimes a)= x\cdot a$, we have:
		$$\theta (x\otimes a) = x_{(1)}\cdot a \otimes x_{(2)}, \qquad x\in U, \ a\in A.$$
		Since $\theta$ is a twisting map we get that $\mu$ defines a left $U$-module algebra structure on $A$. Then, again, the commutativity of Diagram (\ref{eqn:diagram-galois}) ensures, by  Lemma \ref{lem:compatsemi}, that the right coaction $\rho_A : A \to A\otimes H$ is left $U$-linear, and this finishes the proof.
	\end{proof}

\begin{theorem}\label{thm:hopfgaloissemidirect}
	Let  $E = H\rtimes U$ be a semi-direct product Hopf algebra, with $U$ cocommutative.
	\begin{enumerate}
		\item Let $L$ be a Hopf algebra, let $A$ be an $(L,H)$ bi-Galois object, and let $\sigma : U\otimes U \to k$ be a $2$-cocycle. Assume that $U$ acts on $A$ as a $U$-module algebra and that the right coaction $\rho_A : A \to A\otimes H$ is left $U$-linear, so that $A\# {_\sigma U}$ is right $H\rtimes U$-Galois object. Then the associated left Hopf algebra $F = L(A\# {_\sigma U}, H\rtimes U)$ is a semi-direct Hopf algebra $L\rtimes U$, where the action of $U$ on $L$ is such that the left coaction $\lambda_A : A \to L\ot A$ is left $U$-linear.
		\item Conversely, any Hopf algebra $F$ such that there exists an $F-H\rtimes U$  bi-Galois object satisfies $F\simeq L \rtimes U$, for $U$ acting on $L$ as above.
	\end{enumerate}
\end{theorem}

\begin{proof} (1) We know from Theorem \ref{thm:bicrossedtwist2} that $F\simeq L\btb U^\sigma$ and we have to prove that the right action $\btl \colon U^\sigma \ot L\to U^\sigma$  is trivial. Recall also from Theorem \ref{thm:bicrossedtwist2} that the following  commutes:
		\begin{align*}
	\xymatrix{ U\ot A \ar[rr]^\theta  \ar[d]^{ \lambda_{U \ot A}}& & A\ot U \ar[d]^{\lambda_{A\ot U}}\\
		U\ot L\ot  U\ot A  \ar[rr]^{\ \omega_{\btb}\ot\theta}& & L\ot U\ot A \ot U}
	\end{align*}	
	where we omit the $\sigma$ subscripts since only coalgebras and comodules structures are involved in the diagram. We thus have, for $x\in U$ and $a \in A$,
	$$\omega_{\btb}(x_{(1)}\ot a_{(-1)}) \ot x_{(2)}\cdot a_{(0)} \ot x_{(3)}= (x_{(1)}\cdot a)_{(-1)}\ot x_{(2)} \ot (x_{(1)}\cdot a)_{(0)} \otimes x_{(3)}$$
	and hence
		\begin{align}\label{eqn:diag} 
			x_{(1)} \btr a_{(-2)}\ot x_{(2)} \btl a_{(-1)}  \ot x_{(2)}\cdot a_{(0)} \ot x_{(3)}= (x_{(1)}\cdot a)_{(-1)}\ot x_{(2)} \ot (x_{(1)}\cdot a)_{(0)} \otimes x_{(3)}.
		\end{align}
	Applying the counit to the first factor, this gives
	\[
x_{(1)} \btl a_{(-1)}  \ot x_{(2)}\cdot a_{(0)} \ot x_{(3)}= x_{(2)} \ot x_{(1)}\cdot a \otimes x_{(3)}.	
	\]
	Applying the antipode of $U$ on the right term and making it act on the middle one, gives by the cocommutativity of $U$
\[
 x\btl a_{(-1)}  \ot a_{(0)} = x \ot  a.
\]
For $l \in L$, let $\sum_i a_i \otimes b_i \in A\ot A$ be such that $l\otimes 1 = \sum_ia_{i(-1)}\ot a_{i(0)}b_i$ ($A$ is left $A$-Galois). We then have
\begin{align*}
x \btl l \ot 1 & = \sum_i x\btl a_{i(-1)}\ot a_{i(0)}b_i = \sum_i x \otimes a_ib_i = \varepsilon(l) x \otimes 1
\end{align*}		
and this proves that the right action $\btl$ is indeed trivial. Another glance at Equation \ref{eqn:diag} then shows that the left coaction $\lambda_A : A \to L\ot A$ is left $U$-linear, as claimed.

(2) follows  from Theorem \ref{thm:galoissemidirect}, from the first assertion,  and from the uniqueness of the left Hopf algebra associated to a right Galois object.
\end{proof}	

In the situation with a trivial cocycle,  Theorem \ref{thm:galoissemidirect} gives:

\begin{corollary}
Let   $E = H\rtimes U$ be a semi-direct product Hopf algebra, let $A$ be a $L$-$H$-bi-Galois objets and assume that $U$ acts on $A$ and $L$ as $U$-module algebras  and that the right coaction $\rho_A : A \to A\otimes H$ is left $U$-linear and the left coaction $\lambda_A : A \to L\ot A$ is left $U$-linear. Then the smash product $A\# U$ is a $L\rtimes U$-$H\rtimes U$-bi-Galois object.
\end{corollary}

\begin{example}
Let $U=k\lg t | t^2=1 \rg$ be the group algebra of the cyclic group $C_2$  and consider matrices ${\bf p}=(p_{ij}), {\bf q}=(q_{ij})\in M_n(k)$ with $p_{ii}=1=q_{ii}$ and $p_{ij}p_{ji}=1=q_{ij}q_{ji}$ for every $1\leq i,j\leq n$.	
The algebra $\Oo_{{\bf p},{\bf q}}=\Oo_{{\bf p},{\bf q}}(\GL_n(k))$ associated to such a pair of matrices ${\bf p},{\bf q}$ is defined by $n^2$ generators $x_{ij}, y_{ij}$, $1\leq i,j\leq n$ and relations 
	\begin{align*}
		x_{kl}x_{ij}&=p_{ki}q_{jl}x_{ij}x_{kl}, & y_{kl}y_{ij}&=p_{k,i}q_{jl}y_{ij}y_{kl}, & y_{kl}x_{ij}&=p_{ik}q_{lj}x_{ij}y_{kl},
	\end{align*} 
	together with $\sum_{k=1}^{n}x_{ik}y_{jk}=\delta_{i,j}=\sum_{k=1}^{n}x_{ki}y_{kj}$.
	
	In this way $\Oo_{\bf p}\coloneqq \Oo_{{\bf p},{\bf p}}$ and $\Oo_{\bf q}\coloneqq \Oo_{{\bf q},{\bf q}}$ are Hopf algebras and $\Oo_{{\bf p},{\bf q}}$ is a $(\Oo_{{\bf p}},\Oo_{{\bf q}})$-bi-Galois object \cite[3.4]{bic}, with comultiplication, resp.~coaction, induced by the comatrix comultiplication
	\[
	x_{ij}\mapsto \sum_{k=1}^{n}x_{ik}\ot x_{kj}, \qquad y_{ij}\mapsto \sum_{k=1}^{n}y_{ik}\ot y_{kj}.
	\]
	The antipode in these Hopf algebras is determined by $S(x_{ij})=y_{ji}$, $S(y_{ij})=x_{ij}$.
	
	%\smallbreak
	
	Then $U$ acts as an $U$-module algebra on all $H=\Oo_{\bf q}$, $A=\Oo_{{\bf p},{\bf q}}$ and $L=\Oo_{\bf p}$ by switching $x_{ij}\leftrightarrow y_{ij}$, and it is immediate that the involved coactions are $U$-linear. Hence $\Oo_{\bf p, \bf  q}\# U$ is a $\Oo_{\bf p}\rtimes U$-$\Oo_{\bf q}\rtimes U$-bi-Galois object
	
	%In other words, $(H,U,\tr,\tl)$ is matched pair of Hopf algebras if we choose $U\tl H\to U$ to denote the trivial action and define 
	%\[
	%t\tr x_{ij}= y_{ij}, \qquad\qquad  	t\tr y_{ij}= x_{ij}, \qquad 1\leq i,j\leq n.
	%\]
	%We mentioned that $A$ is a $(L,H)$-Galois object. We let $Q\coloneq U$ as Hopf algebras and consider the algebra $R\coloneqq U$ as a (trivial) $(U,Q)$-Galois object. 
	%Now we can define a twisting map $\theta\colon R\to A\to A\ot R$ by setting 
	%\[
	%\theta(x_{ij},a)=y_{ij}\ot t, \qquad\qquad  \theta(y_{ij},a)=x_{ij}\ot t, \qquad 1\leq i,j\leq n.
	%\]
	%It is readily seen that diagram \eqref{eqn:diagram-galois} commutes; observe that associated twisting map $\omega_{\tb}\colon U\ot H\to H\ot U$ is $t\ot x_{ij}\mapsto y_{ij}\ot t$, $t\ot y_{ij}\mapsto x_{ij}\ot t$.
	
	%Hence we can apply Theorem \ref{thm:bicrossedtwist2} to recover $F=L(A\#_\theta R,H\tb U)$ as a bicrossed product $L\btb Q$, for the trivial action $\btl \colon Q\ot L\to Q$ and $\btr \colon Q\ot L\to L$ given by 
	%\[
	%t\btr x_{ij}= y_{ij}, \qquad\qquad  	t\btr y_{ij}= x_{ij}, \qquad 1\leq i,j\leq n.
	%\]
	%The map $\omega_{\btb}\colon Q\ot L\to  L\ot Q$ can be easily computed from this.
\end{example}

%\comj{Write down (not necessarily as an example) the case where $U$ is a group algebra, or an enveloping  algebra of a Lie algebra $\rightarrowtail$ "unrolled quantum group" $U_q(\mathfrak{sl}_2)\rtimes k[X]$, \cite{cgp}. }

%\comy{Created a subsection for this}

\subsection{Unrolled Hopf algebras}\label{sec:unrolled}
A particular instance where a semidirect product $H\rtimes U$ as in Theorem \ref{thm:galoissemidirect} is considered is the setting of \textit{unrolled Hopf algebras}: here $H$ is a Hopf algebra and  $U=U(\g)$, where $\g$ is a Lie algebra acting on $H$ by $k$-biderivations (i.e.~by endomorphisms that are both derivations and co-derivations); see \cite{as} for details.

 We may re-brand Theorem \ref{thm:galoissemidirect} into the following.

\begin{corollary}\label{cor:unrolled}
Let $H$ and $\g$ as above and assume there is a right $H$-Galois object on which $\g$ acts by $k$-biderivations. Let $F=L(A,H)$ and let $\sigma : U(\g)\otimes U(\g) \to k$ be a $2$-cocycle. Then the left Hopf algebra $L(A\#_\sigma U,H\rtimes U)$ is an unrolled Hopf algebra $F\rtimes U(\g)$. Conversely, any Hopf algebra $F$ for which there is an $F-H\rtimes U(\g)$ bi-Galos object is an unrolled Hopf algebra.\qed
\end{corollary}

%\subsubsection{Example: unrolled quantum $\sl_2$}
We now specialize to the context of unrolled quantum $\sl_2$ in \cite{cgp}, where $H$ is (a cover of) the small quantum group $\mathfrak{u}_q(\sl_2)$ and $U=k[X]$ (namely $\dim \g=1$).
As an illustration, we revisit this situation within our setup. 

We fix $1<\ell\in\N$ and  $q\in k$ a primitive $2\ell$th root of 1. We consider the Hopf $k$-algebra $H\coloneqq \bar{U}_q(\sl_2)$ generated by $E,F,K^{\pm1}$ and relations
\begin{align*}
K^{\pm1}K^{\mp1}&=1, & KE&=q^2EK, & KF&=q^{-2}FK, & EF-FE&=\frac{K-K^{-1}}{q-q^{-1}}, & E^\ell&=0, &F^\ell&=0.
\end{align*}

Let $U=k[X]$ denote the polynomial algebra on a variable $X$ (or the enveloping algebra of the one-dimensional Lie algebra $\g=k$) and let $k[X]$ act on $H$ via
\begin{align}\label{eqn:action-unrolled}
	X\tr K^{\pm1}&=0, & X\tr E&=2E, & X\tr F&=-2F.
\end{align}
This defines a semidirect product Hopf algebra $H\rtimes k[X]$, with commutators:
\begin{align*}
	[X,K]=0, \qquad [X,E]=2E, \qquad [X,F]=-2F.
\end{align*} 
This algebra is denoted $\bar{U}_q^X(\sl_2)$ in \cite{cgp} and called the \textit{unrolled quantum group of $\sl_2$}. 

We observe that the algebra $H=\bar{U}_q(\sl_2)$ fits into a chain of quotients
\begin{align*}
	U_q(\sl_2)\twoheadrightarrow H\twoheadrightarrow \mathfrak{u}_q(\sl_2)
\end{align*}
as  $H=U_q(\sl_2)/\lg E^\ell,F^\ell\rg$ and $\mathfrak{u}_q(\sl_2)=H/\lg K^{2\ell}-1\rg$. 
The cleft objects for the left and right ends of this sequence were computed in \cite[Lemma 16]{g}  and \cite[Lemma 25]{g}, respectively. 

The same tool developed therein, particularly \cite[Theorem 8]{g} together with its systematization \cite[\S5]{aagmv}, shows that the cleft objects for $H$ are the algebras $A_{(a,b,c)}$, $a,b,c\in k$, generated by $e,f,g^{\pm1}$ so that
\begin{align*}
	g^{\pm1}g^{\mp1}&=1, & ge&=q^2eg, & gf&=q^{-2}fg, & ef-fe&=ag-\frac{1}{q-q^{-1}}g^{-1}, & e^\ell&=b, & f^\ell&=c.
\end{align*}
The coaction $\rho\colon A_{(a,b,c)}\to A_{(a,b,c)}\ot H$ is as expected. Now \cite[Theorem 12]{g}, also \cite[\S5.5]{aagmv}, shows that the associated left Hopf algebra $L_{(a,b,c)}\coloneqq L(A_{(a,b,c)},H)$ is the algebra generated by $K^{\pm1}, E, F$ with commutation  $K^{\pm1}K^{\mp1}=1$, $KE=q^2EK$, $KF=q^{-2}FK$ and  the {\it liftings}:
\begin{align*}
EF-FE&=a\frac{K-K^{-1}}{q-q^{-1}}, & E^\ell&=b(1-K^{2\ell}), &F^\ell&=c(1-K^{2\ell}).
\end{align*}
We remark that $L_{(a,0,0)}\simeq H$ when $a\neq 0$ and $L_{(0,0,0)}\simeq k[E,F]\#k\Z$ is graded.

It is easy to check that for each $\lambda\neq 0$, the assignment  $X^n\ot X^m\mapsto \delta_{n,m}n!\lambda^n$ defines a Hopf cocycle $\sigma=\sigma_\lambda\colon U\ot U\to k$; moreover these are all such cocycles. The associated cleft object ${}_\sigma U$ is the vector space $k[X]$ with multiplication $X^rX^s=\sum\limits_{i=0}^{\min\{r,s\}}\lambda^i\binom{r}{i}\binom{s}{i}i!X^{r+s-2i}$.

\smallbreak

The following is a direct consequence of Theorem \ref{thm:galoissemidirect} and Corollary \ref{cor:unrolled}.
\begin{example}
Let us set $H=\bar{U}_q(\sl_2)$ and $U=k[X]$, as above. 
\begin{enumerate}
\item An $H$-Galois object $A_{(a,b,c)}$ is an $U$-module algebra in such a way that the coaction $\rho\colon A_{(a,b,c)}\to A_{(a,b,c)}\ot H$ is $U$-linear
 if and only if $b=c=0$. The action is
\begin{align*}
	X\cdot g^{\pm1}&=0, & X \cdot e&=2e, & X\cdot f&=-2f.
\end{align*}
\item The associated smash product algebra $A_{(a,0,0)}^X=A_{(a,0,0)}\#_\sigma U$ is a right $\bar{U}_q^X(\sl_2)$-Galois object and $L(A_{(a,0,0)}^X,\bar{U}_q^X(\sl_2))$ is a semidirect product $L_{(a,0,0)}^X=L_{(a,0,0)}\rtimes k[X]$, for $k[X]$ acting on $L_{(a,0,0)}$ via \eqref{eqn:action-unrolled}.
\end{enumerate}
It follows that $L_{(a,0,0)}^X\simeq \bar{U}_q(\sl_2)^X$ when $a\neq 0$.
\pf
Remember that the coinvariants $A_{(a,b,c)}^{\co H}=\{a\in A:\rho(a)=a\ot 1\}$ are trivial, namely $A_{(a,b,c)}^{\co H}\simeq k$. This implies that there is $\mu\in k$ so that $X\cdot g^{\pm1}=\pm\mu g^{\pm1}$. This fact, together with the linearity of $\rho$, sets $X\cdot e=2(\mu+1)e$ and $X\cdot f=-2(\mu+1)f$. As $X\cdot e^\ell=2\ell(\mu+1)e$, similarly for $f$, we get that either $\mu=-1$ or $b=c=0$. However, as $X\cdot (ef-fe)=0$, we get that $\mu=0$ and therefore (1) follows. (2) is Corollary \ref{cor:unrolled}.
\epf
\end{example}

%\section{Smash product of a Hopf-Galois object by a Hopf algebra}

%\comy{Removed.}

\bibliographystyle{amsalpha}

\end{document}